\documentclass[10pt,twocolumn,twoside]{IEEEtran}

\usepackage{amsmath,amsfonts,amssymb,amscd,latexsym,enumerate,bbm}
\usepackage{theorem}
\usepackage{arydshln}

\usepackage{tikz}
\usepackage{tkz-berge}
\usepackage{dsfont}

\usepackage{enumerate}
\usepackage{bm}
\usetikzlibrary{calc,positioning,patterns,decorations.pathmorphing,decorations.markings}
\tikzstyle{branch} = [circle,inner sep=0pt,minimum size=1mm,fill=black,draw=black]

\usepackage{tikz}
\usepackage{tkz-berge}
\usetikzlibrary{positioning}
\usetikzlibrary{arrows,%
	petri,%
	topaths}%
\usetikzlibrary{decorations.markings}         
\tikzstyle{vertex}=[circle, shading = ball, ball color = white!100!white, minimum size = 15pt, draw, inner sep=0pt]  
\newcommand{\vertex}{\node[vertex]}           

\DeclareMathOperator{\proj}{proj}
\DeclareMathOperator{\cl}{cl}

{\theorembodyfont{\itshape}\newtheorem{theorem}{Theorem}}

{\theorembodyfont{\itshape}\newtheorem{lemma}[theorem]{Lemma}}
{\theorembodyfont{\itshape}\newtheorem{proposition}[theorem]{Proposition }}
{\theorembodyfont{\itshape}}
{\theorembodyfont{\upshape}\newtheorem{definition}[theorem]{Definition}}
{\theorembodyfont{\itshape}}
{\theorembodyfont{\upshape}\newtheorem{remark}[theorem]{Remark}}
{\theorembodyfont{\upshape}}
{\theorembodyfont{\upshape}}
{\theorembodyfont{\upshape}\newtheorem{assumption}{Assumption}}

\makeatletter
\DeclareRobustCommand{\qed}{%
	\ifmmode 
	\else \leavevmode\unskip\penalty9999 \hbox{}\nobreak\hfill
	\fi
	\quad\hbox{\qedsymbol}}
\newcommand{\openbox}{\leavevmode
	\hbox to.77778em{%
		\hfil\vrule
		\vbox to.675em{\hrule width.6em\vfil\hrule}%
		\vrule\hfil}}
\newcommand{\qedsymbol}{\openbox}
\newenvironment{proof}[1][\proofname]{\par
	\normalfont
	\topsep6\p@\@plus6\p@ \trivlist
	\item[\hskip\labelsep\itshape
	#1.]\ignorespaces
}{%
	\qed\endtrivlist
}
\newcommand{\proofname}{Proof}
\makeatother

\def\calx{{\cal X}}
\def\cali{{\cal I}}
\def\caln{{\cal N}}
\def\calt{{\cal T}}
\def\cals{{\cal S}}
\def\cale{{\cal E}}

\DeclareMathAlphabet{\mymathbb}{U}{bbold}{m}{n}

\newcommand{\col}{\mbox{col}}
\newcommand{\diag}{\mbox{blockdiag}}
\newcommand{\bone}{\mathds{1}}
\newcommand{\bze}{\mymathbb{0}}
\newcommand{\bx}{\bm{x}}
\newcommand{\by}{\bm{y}}

\newcommand{\bd}{\bm{d}}
\newcommand{\bsg}{\bm{\sigma}}

\newcommand{\bpsi}{\bm{\psi}}
\newcommand{\bphi}{\bm{\phi}}
\newcommand{\bxi}{\bm{\xi}}
\newcommand{\bpi}{\bm{\Pi}}
\newcommand{\bnu}{\bm{\nu}}
\newcommand{\ba}{\bm{A}}

\newcommand{\bk}{\bm{K}}
\newcommand{\bt}{\bm{H}}

\newcommand{\avg}[1]{s(#1)}

\newcommand{\R}{\mathbb{R}}

\newcommand{\fpart}[2]{ \frac{\partial #1}{\partial #2}}

\DeclareMathOperator*{\argmin}{arg\,min}
\DeclareMathOperator{\ima}{im}

\usepackage[prependcaption,colorinlistoftodos]{todonotes}

        \usepackage{setspace}
\usepackage{url,hyperref}

\author{Mehran Shakarami, Claudio De Persis, and Nima Monshizadeh%
	\thanks{Mehran Shakarami, Claudio De Persis, and Nima Monshizadeh are with the Engineering and Technology Institute, University of Groningen, 9747AG, The Netherlands, {\tt\small m.shakarami@rug.nl, c.de.persis@rug.nl, n.monshizadeh@rug.nl.}}}
\begin{document} 
 \title{{Privacy and Robustness Guarantees in Distributed Dynamics for Aggregative Games}}
\date{}
\maketitle

\begin{abstract}
	This paper considers the problem of Nash equilibrium (NE) seeking in aggregative games, where the payoff function of each player depends on an aggregate of all players' actions.  We present a distributed continuous time  algorithm such that  the actions of the players converge to NE by communicating to each other through a connected network. A major concern in communicative schemes among strategic agents is that their private information may be revealed to other agents or to a curious third party who can eavesdrop the communications. We address this concern for the presented algorithm and show that
		private information of the players cannot be reconstructed even if all the communicated variables are compromised. 
		As agents may deviate from their optimal strategies dictated by the NE seeking protocol, we investigate robustness of the 
proposed algorithm against time-varying disturbances. In particular, we provide rigorous robustness guarantees by proving input to state stability (ISS) properties of the NE seeking dynamics.   
		Finally, we demonstrate practical applications of our theoretical findings on two case studies; namely, on an energy consumption game and  a charging coordination problem of electric vehicles.
\end{abstract}
\section{Introduction}
Game theory is the standard tool for studying the interaction behavior of self-interested agents/players  and has attracted considerable attention 
due to its broad applications and technical challenges. An active research topic in this area {concerns} aggregative games that model a set of  noncooperative  agents aiming at minimizing their cost functions, while the action of each individual player is influenced by an aggregation of the actions of all the other players \cite{jensen2010aggregative}. 
The most notable example of aggregative games is the Cournot  competition in {economics} \cite{mas1995microeconomic}.  {These} games have appeared in a broad range of applications such as networked control systems \cite{de2019feedback}, demand-side management in smart grids \cite{mohsenian2010autonomous}, charging control of plug-in electric vehicles \cite{ma2011decentralized}, and flow control of communication networks \cite{alpcan2005distributed}.


Existence of a solution for games, Nash equilibrium (NE), and its uniqueness have been extensively studied in the literature, and various NE computation algorithms have been proposed \cite{basar1999dynamic}. Earlier {works} considered the case where each agent has full access to  the actions of all other agents, i.e., all-to-all interactions  \cite{basar1999dynamic,facchinei2007finite}. However, recent works have attempted  to relax this assumption due to computational and scalability issues. In this regard,  {the authors in \cite{salehisadaghiani2016distributed,ye2017distributed,gadjov2018passivity} presented distributed NE seeking algorithms  where each player computes an estimation of the actions of all the other players by communicating to its neighbors. Although those algorithms are applicable to aggregative games, they are inefficient  as they require that each player estimate the  action of all other players. In aggregative games, on the other hand,  it is sufficient that each player estimates the aggregation term. This has led to various algorithms tailored for aggregative games,}  which can be  classified as  {gather and broadcast}  \cite{de2018continuous,grammatico2017dynamic,paccagnan2018nash} and distributed algorithms  \cite{koshal2016distributed,parise2017distributed,Gadjov2019single,liang2017distributed}. 
{The former is based on the exchange of information with a central aggregator, whereas the latter relies on a peer-to-peer communication.}
This paper falls into the second category and presents a fully distributed NE seeking algorithm for aggregative games.

{From a different perspective, distributed NE seeking algorithms for aggregative games can be divided into discrete time \cite{koshal2016distributed,lei2018linearly,parise2017distributed,Gadjov2019single} and continuous time  \cite{liang2017distributed,ye2016game}. The discrete time algorithms are based on best response dynamics \cite[Synchronous Alg.]{koshal2016distributed}, gossip technique \cite[Asynchronous Alg.]{koshal2016distributed}, double-layer iterations \cite{lei2018linearly,parise2017distributed}, and forward-backward iteration \cite{Gadjov2019single}. The continuous time algorithms are based on best response dynamics (gradient based algorithms). The asynchronous algorithm in \cite{koshal2016distributed}, needs diminishing step sizes for exact convergence, which typically slows down the convergence speed, and if a fixed step size is used, the solution will only converge to a vicinity of the NE. The other algorithms, on the other hand, employ some tuning parameters shared among all players.  In comparison to those works, we provide fully distributed conditions for implementing the proposed NE seeking algorithm, and more importantly, equip our algorithm with rigorous  {privacy and robustness} guarantees as discussed below.}

Generally speaking, NE seeking algorithms rely on communication either with a central aggregator \cite{de2018continuous} or among neighboring agents \cite{koshal2016distributed}.  In the former {approach}, it is often assumed that the aggregator is trustworthy, whereas, in reality, private information can still be leaked by an aggregator either willingly or unwillingly.
In the latter {approach}, private information can be revealed to other players through direct communication, or leaked to curious adversaries as a result of eavesdropping. 
More generally, in order to convince strategic players to participate in any cooperative policy, privacy guarantees need to be put in place.  

Motivated by the above concerns, we investigate the proposed distributed NE seeking algorithm from the viewpoint of privacy.   To this end, we adopt the notion of privacy recently proposed in  \cite{Monshizadeh2019plausible}. Roughly speaking, privacy is preserved if private variables of the dynamics cannot be reconstructed based on the available information on the structure of the algorithm, the class of payoff functions, measurements, and communicated variables. To make sure this is the case, it will be shown that there are {\em replicas} of private variables that are indistinguishable from the original ones in view of the available information. 
An alternative approach would be to use data perturbation techniques and rely on differential privacy 
\cite{dwork2011differential,dwork2014algorithmic,Cortes2016}.  The idea behind this technique is to add noise with appropriate statistical properties to the process under investigation in order to limit the ability of a curious party in estimating the private quantities of the system.
However, the added noise will  steer the asymptotic behavior of the algorithm away from the NE of the game .
Our approach, on the contrary, retains the NE of the game while providing privacy guarantees.

The payoff functions do not capture all practical features of a game, due to  the underlying approximations in obtaining the payoff functions {or unidentified parameters. Therefore, it is crucial that an NE seeking algorithm has suitable robustness properties. More importantly, {robustness} is essential  due to possible deviation of the players actions from a fully rational behavior,  examples of which are ``stubborn players''  \cite{ye2016game,frihauf2011nash} who do not fully obey the NE seeking  dynamics, or ``almost" rational  players whose decisions are determined by their ``bounded rationality'' \cite{pita2010robust}. {Robustness} of an NE seeking algorithm with respect to slowly-varying channel gain  in code division multiple access systems is studied in \cite{fan2006passivity}. {Interested readers are referred to \cite{ito2012disturbance,cherukuri2017role,weitenberg2018exponential} for  studies on robustness of gradient systems,  saddle-point dynamics, and optimal frequency regulation of power networks, respectively.}}


{To cope with the imperfections in the payoff functions, we add  bounded time-varying disturbances to the dynamics of the algorithm}, and {it will be shown} that the proposed distributed NE seeking algorithm is robust against such perturbations. {We  use input-to-state stability} (ISS) as a notion of robustness, which {examines} whether the state trajectories of the system  are bounded by a function of the  perturbation \cite{sontag2008input}.

In summary, the main contribution of the paper is threefold.
First, we provide {a fully distributed algorithm in continuous-time  that steers the players to the NE.} {Second, we provide privacy guarantees for the proposed algorithm. Third, we show robustness of the algorithm {in the sense of ISS} against bounded time-varying disturbances.}

The rest of the paper is organized as follows: Section~\ref{problem} includes preliminaries and the problem formulation. In Section~\ref{alg_analysis}, a distributed NE seeking algorithm is proposed and its {privacy and robustness}   guarantees are established. The algorithm is modified in Section~\ref{projection} to {deal with} the case when the action of each player is constrained to a compact set. {Two case studies of an energy consumption game and charging of electric vehicles are} provided in Section~\ref{simulation}.  The paper closes with conclusions in Section~\ref{conclusion}.

\section{Notations, Preliminaries, and Problem Statement}\label{problem}
\subsection{Notations}
The set of real, positive real, and nonnegative real numbers are denoted by $ \R $, $ \R_{>0} $, and $ \R_{\geq0} $, respectively. We use $ \bze$ to denote a  vector or matrix of all zeros.  The symbol $ \bone_n $ denotes the vector of all ones in $\R^n$, and $ I_n $ denotes the identity matrix of size $n$. We omit the subscript whenever no confusion arises. The Kronecker product is denoted by $ \otimes $. For given vectors $ x_1,\cdots,x_N \in\R^n$,  we use the shorthand notation  $\bx:=\col\big(x_1,\cdots,x_N\big)=\big[x_1^\top,\cdots,x_n^\top\big]^\top $ and  $ \bx_{-i}:=\col\big(x_1,\cdots,x_{i-1},x_{i+1},\cdots,x_N\big) $.
We  use $ \bm{A}:=\diag\big(A_1,\cdots,A_N\big) $ to denote the block diagonal matrix constructed from the matrices $ A_1,\cdots,A_N $. A continuous function $ \alpha:\R_{\geq 0}\to \R_{\geq 0}$ is class $ \mathcal{K} $ if it is strictly increasing and $ \alpha(0)=0 $. In addition, it is class $ \mathcal{K}_\infty $ if $ \alpha(s)\to \infty $ as $ s\to \infty $. A continuous function $ \beta:\R_{\geq 0}\times \R_{\geq 0}\to \R_{\geq 0} $ belongs to class $ \mathcal{KL} $ if for any fixed $ t $, the mapping $ s\mapsto\beta(s,t) $ belongs to class $ \mathcal{K} $, and for any fixed $ s $, the mapping $ t\mapsto\beta(s,t) $ is decreasing and $ \beta(s,t)\to 0 $ as $ t\to\infty $. A function $ F:\R^n\to \R^n $ is (strictly) monotone if $ (x-y)^\top(F(x)-F(y))\geq 0\,(>0) $ for all $ x\neq y\in \R^n $, and it is $ \mu $-strongly monotone if $ (x-y)^\top(F(x)-F(y))\geq \mu \|x-y\|^2 $ for all $ x, y\in \R^n $ and some $ \mu\in \R_{>0} $.
\subsection{Algebraic Graph Theory}
Let $ G_c=(\cali,\cale) $ be an undirected graph that models the network of $ N $ agents with $ \cali=\{1,\cdots,N\} $ being the node set associated to the agents, and $ \cale$ denoting the edge set. Each element of $ \cale $ is an unordered pair $ \{i,j\} $ with $ i,j\in\cali $. The graph is connected if there is a path between every pair of nodes. The set of neighbors of agent $ i $ is $ \caln_i=\{j\in\cali\mid \{i,j\}\in\cale\} $. The Laplacian matrix of $ G_c $ is denoted by $ L $ with $ L_{ii} $ equal to the cardinality of $ \caln_i $, $ L_{ij}=-1 $ if $ j\in\caln_i $, and $ L_{ij}=0 $ otherwise. The matrix $ L $ of an undirected graph is positive semidefinite and $ \bone_{N}\in\ker(L) $. If the graph is connected, $ L $ has exactly one zero eigenvalue, and  $ \ima  (\bone_{N})=\ker(L) $. {The Moore–Penrose inverse of $ L $ is denoted by $ L^{+ } $.}

\subsection{Projection and Variational Inequality}
Given a closed convex set $ \cals \subseteq \R^n$, the projection of a point $ v\in\R^n $ to $ \cals $ is denoted by $ \proj_\cals (v):=\argmin_{y\in\cals}\|y-v\| $. Given a point $ x\in\cals $, the normal cone of $ \cals $ at $ x $ is the set $ \caln_{\cals}(x):=\big\{y\in\R^n \mid y^\top(z-x)\leq 0,\forall z\in\cals\big\} $. The tangent cone of $ \cals $ at $ x\in\cals $ is denoted by $ \calt_\cals(x):=\cl\left(\cup_{y\in\cals}\cup_{h>0} h(y-x)\right) $ where $ \cl(\cdot) $ denotes the closure of a set. For $ v\in\R^n $ and $ x\in\cals $, the projection of $ v $ at $ x $ with respect to $ \cals $ is given by $ \Pi_{\cals}(x,v):=\lim\limits_{h\to 0^{+}}\frac{1}{h}\left(\proj_\cals(x+hv)-x\right) $, and it is equivalent to the projection of $ v $ to $ \calt_\cals(x) $, i.e., $ \Pi_{\cals}(x,v)=\proj_{\calt_\cals(x)}(v) $. By using Moreau's decomposition theorem, a vector $ v\in \R^n $ can be decomposed as $ v=\proj_{\caln_\cals(x)}(v)+\proj_{\calt_\cals(x)}(v) $ for any point $ x\in \cals $. Given a mapping $ F:\cals\to \R^n $, the variational inequality problem VI$ (\cals,F)$ is to find the point $ \bar x\in \cals $ such that $ (x-\bar x)^\top F(\bar x)\geq 0 $ for all $ x\in \cals $.

\subsection{Aggregative Games}
We consider  $ N $ players that can choose their action variables $ x_i $ in the constraint sets $ \calx_i\subseteq \R^n $. In an aggregative game, each player aims at minimizing a payoff function $ J_i:\R^n\times \R^n\to \R $ by choosing the action variable $ x_i $. The value of the payoff function depends on $x_i$ and an aggregation of all the other action variables.  In particular, each  player $i\in \cali$ attempts to solve the following minimization problem 
\begin{equation}\label{cost}
\begin{split}
&\min_{x_i\in\calx_i}J_i(x_i,\avg{\bx})\\
&\avg{\bx}:=\frac{1}{N}\sum_{j\in\cali}h_jx_j=\frac{1}{N}(\bone_N^\top \otimes I_n)\bt\bx,
\end{split}
\end{equation}
where $ \cali:=\{1,\cdots,N\} $ is the set of players, $ h_j$ is a positive scalar indicating the weight of the action $ x_j$ in the aggregation {$ \avg{\bx} $}, $\bt:=\diag(h_1I_n,\cdots,h_NI_n) $,  and $ \bx:=\col(x_1,\cdots,x_N) $. Note that the solution of the above problem depends on the action of other players.  We use the compact notation $ {\cal G}_{\text{agg}}=\big({\cal I} ,(J_i)_{i\in \cal I}, (\calx  _i)_{i\in \cal I}\big) $ to denote the aggregative game in \eqref{cost}. By definition, a point $\bx^*:= \col(x_1^*,\cdots,x_N^*)  $ is a Nash equilibrium (NE) of the game if 
{\begin{equation*}
x_i^*\in \argmin_{y\in\calx_i}J_i(y,\frac{h_i}{N}y+\frac{1}{N}\sum_{j\neq i}h_jx_j^*),\quad \forall \, i\in\cali.
\end{equation*}}
This means that at the NE, there is no player that can decrease its payoff by unilaterally changing its action.  {We note that  $ x_i^* $ depends on the optimal action of all the other players, and therefore several coupled optimization problems need to be solved to obtain $ \bx^* $. 
	Consequently, standard distributed optimization techniques cannot be used for solving this problem.} In the next section, we derive local sufficient conditions for existence and uniqueness of NE and present a distributed algorithm that asymptotically converges to this point.
\section{Distributed NE Seeking Dynamics}\label{alg_analysis}
First, we discuss some auxiliary results that are instrumental to prove convergence properties of the NE seeking algorithm proposed later in the section.

\begin{assumption}\label{asmp0}
	For all $ i\in\cali $, the action set is $ \calx_i=\R^n $, and the cost function $ J_i $ is $ \mathcal{C}^2 $ in all its arguments. 
\end{assumption}
{This assumption is similar to {\cite[Asm. 2(i)]{gadjov2018passivity}}, and we will relax} it in {Section \ref{projection}} to any compact and convex subset $\calx_i\subset   \R^n $. However,  $ \calx_i=\R^n $ is considered in this section for clarity of the presentation.

Let $ \sigma_i\in\R^n $ be a local variable associated to each player $i\in \cali$, with the payoff function written as $ J_i(x_i,\sigma_i)$, and define 
\begin{equation}\label{fi}
f_i(x_i,\sigma_i):=\fpart{}{x_i}J_i(x_i,\sigma_i)+\frac{h_i}{N} \fpart{}{\sigma_i}J_i(x_i,\sigma_i).
\end{equation}
It is easy to see that 
\[
 \fpart{}{x_i}J_i(x_i,\avg{\bx})=f_i(x_i,\avg{\bx}). 
 \]
To proceed further, we need the following assumption:

\begin{assumption}\label{asmp1}
	For all $ i\in\cali $, $ x_i\in\calx_i $, and $ \sigma_i\in\R^n $, the  mapping $ x_i\mapsto f_i(x_i,\sigma_i)$ is $ \mu_i $-strongly monotone, and the mapping  $ \sigma_i\mapsto f_i(x_i,\sigma_i)$ is $ \ell_i $-Lipschitz continuous with {$ \mu_i>\ell_i h_i$}.
\end{assumption}

{The assumption above is a decentralized version of \cite[Asm. 1]{de2018continuous} that can be checked locally.} The conditions of  Assumption \ref{asmp1} can be replaced by less conservative, yet more implicit, conditions; see Remark \ref{r:relax}.    


In game theory, it is well-known that the pesudo-gradient mapping defined as $ \col\big((f_i(x_i,\avg{\bx}))_{i\in\cali}\big) $ plays a fundamental role in designing NE seeking algorithms. Motivated  by this and the fact that the players may not have access to $ \avg{\bx} $,  we introduce the following mapping:
\begin{equation}\label{F}
F(\bx, \bsg):=\begin{bmatrix}
\bk\,\col\big((f_i(x_i,\sigma_i))_{i\in\cali}\big)\\
 \bsg-\bt\bx
\end{bmatrix}
\end{equation}
where $ \bk:=\diag(k_1I_n,\cdots,k_NI_n) $ with design parameters $ k_i>0 $,  and $  \bsg:=\col({\sigma}_1,\cdots,{\sigma}_N) $. The following lemma captures some properties of \eqref{F}.
\begin{lemma}\label{pro1}
	Let Assumption~\ref{asmp1} hold and choose $ k_{i} $ such that
	\begin{equation}\label{inter}
	k_{i}\in\big(\frac{(\sqrt{\mu_i}-\sqrt{\mu_i-\ell_ih_i})^{2}}{\ell_i^2},\ \frac{(\sqrt{\mu_i}+\sqrt{\mu_i-\ell_ih_i})^{2}}{\ell_i^2}\big)
	\end{equation}
	is satisfied for each $i\in \cali$.
	Then, for all $ x_i\in\calx_i $ and $ \sigma_i\in\R^n $, 
\begin{enumerate}[(i)]
		\item the map $ F $ in \eqref{F} is $ \epsilon $-strongly monotone.
		\item the map $ \bk\col\big((f_i(x_i,\avg{\bx}))_{i\in\cali}\big) $ is $ \epsilon $-strongly monotone.
\end{enumerate}
\end{lemma}
\begin{proof}
	See Appendix A.
\end{proof}

\begin{remark}
	Setting $ k_i=1 $, for each $i$, returns a more restrictive condition than the one in Assumption~\ref{asmp1}, namely $ \sqrt{\mu_i}>2(\ell_i+h_i) $. Therefore, introducing the gain $k_i$ yields a milder assumption and, as we will see later, contributes to the privacy of the proposed algorithm.
\end{remark}

We note that the results of the preceding lemma is sufficient for the existence and uniqueness of the NE.
This is formally stated next.
\begin{lemma}\label{lemNE}
	Let Assumptions~\ref{asmp0} and \ref{asmp1}  hold. 
	 Then the aggregative game $ {\cal G}_{\text{agg}}=\big({\cal I} ,(J_i)_{i\in \cal I}, (\calx  _i)_{i\in \cal I}\big) $ has a unique NE $ \bx^* $  which satisfies
	\begin{equation}\label{ne}
	\col\big((f_i(x_i^*,\avg{\bx^*}))_{i\in\cali}\big)= \bze
	\end{equation}
with $f_i(\cdot)$ given by \eqref{fi}.
\end{lemma}
\begin{proof}
See Appendix A.
\end{proof}

\begin{remark}\label{r:relax}
	From the presented analysis, one can see that Assumption~\ref{asmp1} can be relaxed  to any payoff function {$ J_i(x_i,\avg{\bx})$} that is strictly convex and radially unbounded  in $ x_i $ for all $ \bx_{-i}\in \calx_{-i} $ {\cite[Asm. 2(i)]{gadjov2018passivity}}, and results in  strong monotonicity of the mapping $ \col(k_i f_i(x_i,\sigma_i),\sigma_i-h_i x_i) $ for some  $ k_i>0 $. 
\end{remark}



For privacy reasons, we assume that the players do not communicate their action variables $x_i$, neither to the other players nor to a central unit.  Instead, auxiliary variables  will be communicated through a connected communication graph $ G_c $.  This motivates the following distributed NE seeking policy:

\begin{equation*}
\begin{aligned}
\dot{x}_i&=-k_if_i(x_i,\sigma_i)\\
\dot{{\sigma}}_i&=-{\sigma}_i+h_ix_i-\sum_{j\in\caln_i}(\psi_i-\psi_j)\\
\dot{\psi}_i&=\sum_{j\in\caln_i}({\sigma}_i-{\sigma}_j),
\end{aligned}
\end{equation*}
for each $i\in \cali$,  where  $ \caln_i $ denotes the set of neighbors of node $ i $. Notice that the players only use the local parameters $ k_i $ and $ h_i $, and communicate the variables $ {\sigma}_i $ and $\psi_i$. The variable $\sigma_i$ is, in fact, a local estimation of $ \avg{\bx} $, {and the state components $ {\psi}_i $, $i\in \cali$,  are defined to enforce consensus on $ \sigma_i$ variables}. 
Let $ \bpsi:=\col(\psi_1,\cdots,\psi_N) $ and $ L $ be the Laplacian matrix of the graph $ G_c $. 
Then, the algorithm can be written in vector form as
\begin{equation}\label{dist}
\begin{split}
\dot{\bx}&=-\bk\col\big((f_i(x_i,\sigma_i))_{i\in\cali}\big)\\
 \dot{\bsg}&=- \bsg+\bt\bx-(L\otimes I_n)\bpsi\\
\dot{\bpsi}&=(L\otimes I_n)\bsg.
\end{split}
\end{equation}
{For clarity, we note that the standing assumption in the remainder of the section is:
	Assumptions~\ref{asmp0} and \ref{asmp1} hold,  and $ k_i $ is selected according to \eqref{inter}, for each $i\in \cali$. }

First, we characterize the equilibria of \eqref{dist} and then proceed with the results concerning convergence, privacy, and robustness. 
\begin{proposition}\label{prop:eqil}
	Let $ \bx^* $ be the NE of the game $ {\cal G}_{\text{agg}} $. Then, any equilibrium point of \eqref{dist} is given by $(\bar\bx,\bar\bsg,\bar\bpsi)= (\bx^*,\bone_N \otimes \avg{\bx^*}, \bar\bpsi) $ where $ \bar \bpsi \in  \varPsi$ with 
	\begin{equation}\label{psi_set}
	\varPsi:=\left\{\bar\bpsi\in\R^{nN}\mid \bar\bpsi= (L^{+}\otimes I_n){\bt}\bx^*+\bone_{N}\otimes\zeta,\, \zeta\in\R^n\right\},
	\end{equation}
	{and $ L^{+ } $ is the Moore–Penrose inverse of $ L $.}
\end{proposition}
\begin{proof}
	At any equilibrium point $ (\bar\bx,\bar\bsg,\bar\bpsi) $, we have
\begin{align}
	\bze&=-\bk\col\big((f_i(\bar x_i,\bar \sigma_i))_{i\in\cali}\big)\label{eqil_x}\\
	\bze&=-\bar \bsg+\bt\bar \bx-(L\otimes I_n)\bar \bpsi\label{eqil_sig}\\
	\bze&=(L\otimes I_n)\bar \bsg.\label{eqil_zet}
\end{align}
As the graph is connected, from \eqref{eqil_zet}, we have $ \bar\bsg=\bone_N\otimes \gamma $ for some $ \gamma\in\R^n $. Therefore, \eqref{eqil_sig} becomes
\begin{equation*}
\bze=-\bone_N\otimes \gamma+\bt\bar \bx-(L\otimes I_n)\bar \bpsi.
\end{equation*}
Left-multiplying both sides of the above equality by $ (\bone_N^\top \otimes I_n) $ gives $ \gamma=\frac{1}{N}(\bone_N^\top \otimes I_n)\bt\bar \bx=\avg{\bar \bx} $. This means that $ \bar\bsg=\bone_N\otimes \avg{\bar \bx} $ and in turn, $ \bar\sigma_i=\avg{\bar \bx} $. Now, \eqref{eqil_x} becomes
\begin{equation*}
\bze=-\bk\col\big(f_i(\bar x_i,\avg{\bar \bx})_{i\in\cali}\big).
\end{equation*}
Consequently, by using Lemma~\ref{lemNE} and $ \bk> 0$, $ \bar\bx $ is the NE of the game, i.e., $ \bar\bx=\bx^* $ and $ \bar\bsg=\bone_N\otimes \avg{ \bx^*} $. In addition,  by substituting the obtained values and using \eqref{pi}, equality \eqref{eqil_sig} yields
\begin{equation*}
(L\otimes I_n)\bar\bpsi= {\bt} \bx^*-\bone_N \otimes \avg{\bx^*}=(\Pi\otimes I_n)\bt\bx^*.
\end{equation*}
Noting that $\Pi=LL^+=L^+L$, we conclude that $ \bar\bpsi $ belongs to the set $ \varPsi $ given by \eqref{psi_set}.
\end{proof}
Proposition~\ref{prop:eqil}, shows that equilibria of \eqref{dist} are crafted as desired, namely  $\bar \bx$ and $\bar \bsg$ return the NE of the game, and the aggregativve value $s(\bx^*)$, respectively.  The next theorem establishes convergence of the solutions of \eqref{dist} to such an equilibrium.

\begin{theorem}\label{th:conv}
	Consider the NE seeking algorithm \eqref{dist} with  initial condition $ ({\bx}(0),{\bsg}(0),\bpsi(0) )\in \R^{nN}\times \R^{nN}\times \R^{nN}$. Then, the solution $ ({\bx},{\bsg},\bpsi) $ converges to the equilibrium point $ (\bar\bx,\bar\bsg,\bar\bpsi)=(\bx^*,\bone_N \otimes \avg{\bx^*}, \bpsi^*) $ where $ \bx^* $  is the unique NE of the aggregative game $ \mathcal{G}_{\text{agg}}$ and 
{
$\bpsi^*\in \Psi$ is given by $\bpsi^*=(L^{+}\otimes I_n) \bt\bx^*+\frac{1}{N} (\bone_N\bone_N^\top \otimes I_n) \bpsi(0)$.
}
\end{theorem}
\begin{proof}
	
	Let $ \tilde{\bx}=\bx-\bar{\bx} $, $ \tilde{\bsg}=\bsg-\bar{\bsg} $, and $ \tilde{\bpsi}=\bpsi-\bar{\bpsi} $, 
	where $(\bar \bx, \bar{\bsg}, \bar \bpsi)$ is an equilibrium of  \eqref{dist}. Note that, by Proposition~\ref{prop:eqil}, we have $\bar\bx=\bx^*$, $ \bar{\bsg}= \bar{\bsg}^*$, and $\bar{\bpsi}\in \Psi$, with $\Psi$ given by \eqref{psi_set}.
Consider the following Lyapunov function candidate 
	\begin{equation*}
	V(\tilde{\bx},\tilde{\bsg},\tilde{\bpsi}):=\frac{1}{2}\|\col (\tilde{\bx},\tilde{\bsg},\tilde{\bpsi})\|^ 2.
	\end{equation*}
	As a result, one can use \eqref{F} and \eqref{dist}   to get
	\begin{equation*}
	\dot V=-\col (\tilde{\bx},\tilde{\bsg})^\top F(\bx, \bsg)-\tilde{\bsg}^\top(L\otimes I_n)\bpsi+\tilde{\bpsi}^\top  (L\otimes I_n) \bsg.
	\end{equation*}
	By adding and subtracting $ \col (\tilde{\bx},\tilde{\bsg})^\top F(\bar\bx,\bar \bsg)$ to the right hand side of {the above equation} and using Lemma~\ref{pro1}$ (i) $, we obtain
	\begin{multline}\label{vdot}
\dot V\leq -\epsilon \|\col (\tilde{\bx},\tilde{\bsg})\|^2	-\col (\tilde{\bx},\tilde{\bsg})^\top F(\bar\bx,\bar \bsg)
\\
-\tilde{\bsg}^\top(L\otimes I_n)\bpsi+\tilde{\bpsi}^\top  (L\otimes I_n) \bsg.
	\end{multline}
Noting  \eqref{F}, \eqref{eqil_x}, and \eqref{eqil_sig}, we have
	\begin{equation*}
	-\col (\tilde{\bx},\tilde{\bsg})^\top F(\bar\bx,\bar \bsg)=- \tilde{\bsg}^\top (\bar \bsg-\bt\bar \bx)=\tilde{\bsg}^\top(L\otimes I_n)\bar \bpsi.
	\end{equation*}
	Consequently,  \eqref{vdot} reduces to
	\begin{align*}
	\dot V&\leq -\epsilon \|\col (\tilde{\bx},\tilde{\bsg})\|^2-\tilde{\bsg}^\top(L\otimes I_n)\tilde{\bpsi}+\tilde{\bpsi}^\top  (L\otimes I_n) \bsg\\
	&=-\epsilon \|\col (\tilde{\bx},\tilde{\bsg})\|^2,
	\end{align*}
	where the last equality is obtained using \eqref{eqil_zet}. To conclude the proof, we use LaSalle's invariance principle. Then, $ ({\bx},{\bsg},\bpsi ) $ converges to the largest invariance set in $ \Omega=\left\{({\bx},{\bsg},\bpsi ) \mid \bx=\bar{\bx},\,  \bsg=\bar{\bsg}\right\} $.	{Consequently, we derive from \eqref{dist} and \eqref{eqil_sig} that 
$\bpsi \in  \Psi$, given by \eqref{psi_set}, on the invariant set. Now, note that  $(\bone^\top \otimes I_n) \bpsi(t)$ is a conserved quantity of the system, and that $\bone^\top L^+=0$. Then, by \eqref{psi_set}, we find that the vector $\bpsi$  converges to $\bpsi^*=(L^{+}\otimes I_n) \bt\bx^*+\frac{1}{N} (\bone_N\bone_N^\top \otimes I_n) \bpsi(0)$, which completes the proof.}
\end{proof}

\subsection{Privacy Analysis}
We resort to the notion of privacy introduced in \cite{Monshizadeh2019plausible} to investigate the privacy of the presented NE seeking algorithm. In particular, privacy is preserved if a curious party cannot uniquely reconstruct the actual private variables of the system. A curious party can be one of the players or an external adversary. 

For technical reasons,  in this section, we restrict the cost functions to 
\begin{equation*}
J_i(x_i,\avg{\bx}):=x_i^\top Q_ix_i+(D_i\, \avg{\bx} +d_i)^\top x_i,
\end{equation*}
where $ Q_i=Q_i^\top\in\R^{n\times n} $, {$Q_i>0$}, $ D_i\in \R^{n\times n} $, and $ d_i\in\R^n $. Note that in this case, the parameters $\mu_i$ and $\ell_i$ in Assumption \ref{asmp1} are given by
\begin{equation*}
\mu_i:=\lambda_{\min}(2Q_i +h_i\frac{D_i+D_i^\top}{2N} ),\quad \ell_i:=\|D_i\|.
\end{equation*}
Let
\begin{align*}
\ba:&=\diag(2Q_i+\frac{h_i}{N}D_i^\top ),\quad  \bm{D}:=\diag(D_i) \\
\bd:&=\col(d_i) , \quad \forall i\in\mathcal{I}.
\end{align*}
Then,  \eqref{dist} reduces to
\begin{equation}\label{distq}
\dot{\bm{\xi}}=\ba_q \bm{\xi}+\bm{D}_q,
\end{equation}
where 
\begin{equation}\label{dist-data}
\begin{split}
\ba_q:&=\begin{bmatrix}
-\bk \ba & -\bk \bm{D} & \bze\\
\bt & -I & -(L\otimes I_n)\\
\bze & (L\otimes I_n) & \bze
\end{bmatrix}\\
\bm{\xi}:&=\col(
{\bx}, {\bsg}, {\bpsi}),\quad \bm{D}_q:=\col(-\bk \bd, \bze,\bze).
\end{split}
\end{equation}

Note that the parameters $ k_i $, $ h_i $, $ Q_i $, $ D_i $, and $ d_i $ are local parameters  associated to each node $i$.
Similarly, the action of each player $x_i$ is local and will be treated as private information. 
On the contrary, both $ {\sigma}_i $ and  $ \psi_i $ are communicated to other agents. 
Therefore,  the latter information is accessible to other players due to direct communication,
or to an adversary as a result of eavesdropping. To provide strong privacy guarantees, we assume that all communicated variables and the Laplacian matrix $L$ are public information, i.e, accessible to a curious party. Such privacy guarantees are valid even if $ N-1 $ players collude to obtain private information of one specific player.  Moreover, the goal and structure of the algorithm are considered public. 
Now, consider a {\em replica} of \eqref{distq} as follows
\begin{equation}\label{replica}
\dot{\bm{\xi}'}=\ba'_q \bm{\xi}'+\bm{D}'_q
\end{equation}
\begin{equation}\label{replica-data}
\begin{split}
\ba_q':&=\begin{bmatrix}
-\bk' \ba' & -\bk' \bm{D}' & \bze\\
\bt' & -I & -(L\otimes I_n)\\
\bze & (L\otimes I_n) & \bze
\end{bmatrix}\\
\bm{\xi}':&=\col(
{\bx}', {\bsg}', {\bpsi}'),\quad \bm{D}_q':=\col(-\bk' \bd', \bze,\bze),
\end{split}
\end{equation}
where the vectors and matrices with ``prime" are defined analogously to the ones without  in \eqref{dist-data}.
Let $(\bx'(t), \bsg'(t), \bpsi'(t))$ be the solution to \eqref{replica} resulting from an initial condition $(\bx'(0), \bsg'(0), \bpsi'(0))$.
Now, we consider the following definition \cite{Monshizadeh2019plausible}: 


\begin{definition}\label{d:privacy}
	The privacy for the algorithm \eqref{distq} is preserved if for any initial condition $ x_i(0)\in\R^{n} $, there exist $ x_i'(0)\in\R^{n}$ such that for each $ i\in\mathcal{I} $ we have
	\begin{equation}\label{unobs}
	{\sigma}_i(t)={\sigma}_i'(t) ,\quad \psi_i(t)= \psi_i'(t),\quad \forall t\geq 0, 
	\end{equation}
	and
	\begin{equation}\label{almost}
	x_i(t)\ne x_i'(t),
	\end{equation}
	for $t=0$ and almost all time $ t> 0$.
\end{definition}

The idea behind the definition is that a curious adversary cannot infer whether the accessible trajectories $\bsg(t)$ and $\bpsi(t)$ are generated from \eqref{distq} with the initial condition $(\bx(0), \bsg(0), \bpsi(0))$ or from \eqref{replica} with the initial condition $(\bx'(0), \bsg(0), \bpsi(0))$.
The resulting confusion limits the ability of an adversary to reconstruct the private quantities and action variables of the players. Note that the qualifier ``almost" in \eqref{almost} is due to the fact that potentially $x_i(t)$ and its replica $x_i'(t)$ can coincide on a set of measure zero. 
Since, we work here with linear dynamics under a constant input $ \bm{D}_q $, the condition \eqref{almost} can be replaced by $x_i(0)\ne x_i'(0)$.
Now, we have the following result: 

\begin{theorem}\label{t:privacy}
	The NE seeking algorithm \eqref{distq} preserves privacy.
\end{theorem}
\begin{proof}
	Defining $\by:=\col({\bsg},\bpsi) $, we have 
	\begin{equation*}
	\by=\bm{C}_q\bm{\xi}, \qquad \bm{C}_q:=	\begin{bmatrix}
	\bze&I&\bze\\
	\bze&\bze&I
	\end{bmatrix}.
	\end{equation*}
	Note that $\by$ contains public information.
	Consider the algorithm \eqref{distq} and its replica \eqref{replica}.
	%
	Analogous to \cite[Prop. 1]{Monshizadeh2019plausible},  privacy is preserved in the sense of Definition \ref{d:privacy} if and only if for any initial condition $ \bm{\xi}(0)$, 
	there exists $\bm{\xi'(0)}$ such that 
	\begin{equation}\label{privcy}
	\begin{split}
	\bm{C}_q\ba_q^k \bm{\xi}(0)&=\bm{C}_q\ba_q'^k \bm{\xi}'(0)\\
	\bm{C}_q\ba_q^k \bm{D}_q&=\bm{C}_q\ba_q'^k \bm{D}_q',
	\end{split}
	\end{equation}
	for all $ k\geq 0 $, and $x_i'(0) \ne  x_i(0)$ for each $i\in \cali$.  
	Note that the above conditions mean that  $\bm{y}=\bm{C}_q\bxi(t)=\bm{C}_q\bxi'(t)$, as desired.
	%
	Verifying \eqref{privcy} for $ k=0 $ results in 
	\begin{equation}\label{prv0}
	{\bsg}(0)=\bsg'(0) ,\, \bpsi(0)=\bpsi'(0).
	\end{equation}
	For $ k=1 $, we obtain 
	\begin{equation*}
	\begin{split}
	\bt {\bx}(0)-{\bsg}(0)-(L\otimes I_n){\bpsi}(0)&= \bt' {\bx}'(0)\\
	&-{\bsg}'(0)-(L\otimes I_n){\bpsi}'(0)\\
	(L\otimes I_n){\bsg}(0)&=(L\otimes I_n){\bsg}'(0)\\
	\bt\bk \bd&=\bt'\bk' \bd'.
	\end{split}
	\end{equation*}
	By using \eqref{prv0}, the above conditions reduce to
	\begin{equation}\label{prv1}
	\bt {\bx}(0)=\bt' {\bx}'(0),\,\bt\bk \bd=\bt'\bk' \bd'.
	\end{equation}
	If we continue this process, it can be seen that the condition \eqref{privcy} becomes
	\begin{equation*}
	\begin{split}
	\bt (\bk \ba)^k {\bx}(0)&=\bt' (\bk' \ba')^k{\bx}'(0)\\
	\bt (\bk \ba)^k \bk \bm{D}&=\bt' (\bk' \ba')^k\bk' \bm{D}'\\
	\bt (\bk \ba)^k \bk \bd&=\bt' (\bk' \ba')^k\bk' \bd',
	\end{split}
	\end{equation*}
	for all $ k\geq 0 $. Note that $ \bt \bk \ba =\bk \ba \bt$ due to block-diagonal structure of the matrices. Therefore, we obtain the following set of equalities
	\begin{equation*}
	\begin{aligned}
	\bt {\bx}(0)&= \bt'{\bx}'(0),\\	
	\bk \ba&=\bk' \ba',	
	\end{aligned}\quad 
	\begin{aligned}
	\bt  \bk \bm{D}&=\bt' \bk' \bm{D}'\\
	\bt  \bk \bd&=\bt' \bk' \bd'.
	\end{aligned}
	\end{equation*}
	Let $\bm{S}_H:={\bm{H}'}^{-1}\bm{H}$ and $\bm{S}_K:={\bm{K}'}^{-1}\bm{K}$. 
	Then, using the commutativity of the involved matrices, the above conditions can be rewritten as
	\begin{equation}\label{privacy1}
	\begin{aligned}
	{\bx}'(0)&= \bm{S}_H\bm{x}(0),\\	
	\ba'&=\bm{S}_K \ba,	
	\end{aligned}\quad 
	\begin{aligned}
	\bm{D}'&= \bm{S}_K\bm{S}_H\bm{D}\\
	\bm{d}'&= \bm{S}_K\bm{S}_H\bm{d},
	\end{aligned}
	\end{equation}
	together with
	\begin{equation}\label{privacy2}
	{\bm{H}'}=\bm{H}\bm{S}^{-1}_H, \qquad {\bm{K}'}=\bm{K}\bm{S}^{-1}_K.
	\end{equation}
	Therefore, an adversary cannot distinguish between the actual system parameters/variables and a replica of the system that satisfies \eqref{privacy1} and \eqref{privacy2}. This completes the proof.
	%
	%
\end{proof}


\begin{remark}
	In order to retain the privacy for games with unweighted action variables,  $ h_i=1 $, a suitable change of variables can be exploited. In particular, each  player $ i\in\cali $ uses a local parameter $ p_i>0$, set $x_i=p_i\hat x_i$, and apply the NE seeking policy  
	\begin{equation*}
	\begin{aligned}
	\dot{\hat x}_i&=-\frac{k_i}{p_i}f_i(p_i\hat x_i,\sigma_i)\\
	\dot{{\sigma}}_i&=-{\sigma}_i+p_i \hat x_i-\sum_{j\in\caln_i}(\psi_i-\psi_j)\\
	\dot{\psi}_i&=\sum_{j\in\caln_i}({\sigma}_i-{\sigma}_j),
	\end{aligned}
	\end{equation*}
	with $k_i$ chosen as before. 
	Clearly, in this case, $\hat x_i(t)$ converges to $p_i^{-1}x_i^*$.
	Then, the NE can be retrieved by multiplying the latter by $p_i$. 
	The fact that the above algorithm preserves privacy can be shown analogous to Theorem \ref{t:privacy}.
\end{remark}


\subsection{Robustness Analysis}
{In this section, we again consider the general form of cost functions $ J_i(x_i,\avg{\bx}) $ and investigate robustness of the dynamical algorithm \eqref{dist} against additive perturbations.} The perturbations can capture uncertainty in the payoff functions, irrationality of the players, or a deliberate addition of noise to improve privacy.\\
 Let $ \bxi:=\col(\bx,\bsg) $,  $ G:=\col(\bze, (L\otimes I_n)) $, and with some abuse of the notation $ F(\bxi):=F(\bx,\bsg) $ {with $ F(\bx,\bsg) $ given by \eqref{F}}. Then, we can rewrite \eqref{dist} with the disturbance $ \bnu(t)\in\R ^{2nN} $ as follows
\begin{equation}\label{dist_dt}
\begin{split}
\dot{\bxi}&=-F(\bxi)-G\bpsi+\bnu\\
\dot{\bpsi}&=G^\top \bxi.
\end{split}
\end{equation}
To analyze  performance of the above algorithm, we resort to the notion of input-to-state stability (ISS) \cite[Def. 4.7]{khalil2002nonlinear} \cite{sontag2008input}. Let $ \tilde{\bxi}:=\bxi-\bar{\bxi} $ and $ \tilde{\bpsi}:=\bpsi-\bar{\bpsi} $ with the equilibrium point $ (\bar{\bxi},\bar{\bpsi}) $ satisfying
\begin{align}
\bze&=-F(\bar\bxi)-G\bar\bpsi\label{eq_dt_1}\\
\bze&=G^\top \bar\bxi.\label{eq_dt_2}
\end{align}
Then, \eqref{dist_dt} is ISS with respect to $ (\bar{\bxi},\bar{\bpsi}) $ if for any $ ({\bxi}(0),{\bpsi}(0))\in \R^{2nN}\times \R^{nN} $ and any  {measurable and locally essentially bounded}   $ \bnu(t) $, the state vector $\col (\tilde{\bxi},\tilde{\bpsi}) $ satisfies
\begin{align*}
\|\col (\tilde{\bxi}(t),\tilde{\bpsi}(t))\| &\leq \beta_0(\|\col (\tilde{\bxi}(0),\tilde{\bpsi}(0))\|,t)\\
&+\beta_1 (\sup_{0\leq \tau \leq t}\|\bnu(\tau)\|),\qquad\qquad \forall t\geq 0,
\end{align*}
where $ \beta_0 $ and $ \beta_1 $ are class $ \mathcal{KL} $ and class $ \mathcal{K} $ functions, respectively.
\begin{theorem}\label{th:iss}
	{Consider the NE seeking algorithm \eqref{dist_dt} with  initial condition $ ({\bxi}(0),\bpsi(0) )\in \R^{2nN}\times \R^{nN}$.	Suppose the disturbance vector $ \bnu(t) $  is measurable and locally essentially bounded,}  {and assume that there exists some positive constant $ \gamma_i $ such that  $ \|\nabla f_i(x_i, \sigma_i)\|\leq \gamma_i $ for all $x_i\in \calx_i$, $\sigma_i\in\R^{n}$, and $i\in \cali$}. Let  $ {\bxi}^*:=\col (\bx^*,\bone_N \otimes \avg{\bx^*}) $ and  $\bpsi^*=(L^{+}\otimes I_n) \bt\bx^*+\frac{1}{N} (\bone_N\bone_N^\top \otimes I_n) \bpsi(0)$ with $ \bx^* $ being the unique NE    of the aggregative game $ \mathcal{G}_{\text{agg}} $. Then,  the NE seeking algorithm \eqref{dist_dt} is ISS with respect to the equilibrium point $({\bxi}^*,{\bpsi}^*)  $.
\end{theorem}
\begin{proof}
From \eqref{dist_dt}, \eqref{eq_dt_1}, and \eqref{eq_dt_2}, we obtain
\begin{equation}\label{e:inc-dynamics}
\begin{split}
\dot{\tilde\bxi}&=-(F(\bxi)-F(\bar\bxi))-G\tilde{\bpsi}+\bnu\\
\dot{\tilde{\bpsi}}&=G^\top \tilde\bxi.
\end{split}
\end{equation}
Define $ \bpi:=\Pi\otimes I_{n} $ and {$ \tilde{\bphi}:=\bpi \tilde \bpsi$}. Then,  we have
\begin{equation}\label{dist_dt_er2}
\begin{split}
\dot{\tilde\bxi}&=-(F(\bxi)-F(\bar\bxi))-G\tilde{\bphi}+\bnu\\
\dot{\tilde{\bphi}}&=G^\top \tilde\bxi,
\end{split}
\end{equation}	
where $ \tilde{\bphi}=\bphi-\bar{\bphi} $ with $ \bar{\bphi}:=\bpi \bar{\bpsi}$, and we have used the fact that $G=G\bpi$. 

The algorithm \eqref{dist_dt_er2} is ISS if for a continuously differentiable function $ V:\R ^{2nN}\times \R^{nN}\to \R $, there exist class $ \mathcal{K}_\infty $ functions $ \alpha_1 $, $ \alpha_2 $, a class $ \mathcal{K} $ function $ \rho $, and a positive definite function $ W(\tilde{\bxi},\tilde{\bphi}) $ such that \cite[Thm. 4.19]{khalil2002nonlinear}
\begin{align}
\alpha_1(\|\col (\tilde{\bxi},\tilde{\bphi})\|)&\leq V(\tilde{\bxi},\tilde{\bphi})\leq \alpha_2(\|\col (\tilde{\bxi},\tilde{\bphi})\|)\label{iss_1}\\
\fpart{V}{\tilde{\bxi}}^\top \dot{\tilde\bxi}+\fpart{V}{\tilde{\bphi}}^\top \dot{\tilde\bphi}&\leq -W(\tilde{\bxi},\tilde{\bphi}),\, \forall \, \|\col (\tilde{\bxi},\tilde{\bphi})\|\geq \rho(\|\bnu \|)>0.\label{iss_2}
\end{align}
Let
\begin{equation}\label{ISS_Function}
V(\tilde{\bxi},\tilde{\bphi}):=\frac{1}{2}\|\col (\tilde{\bxi},\tilde{\bphi})\|^ 2+\kappa \tilde{\bphi}^\top G^\top \tilde\bxi,
\end{equation}
for some $ \kappa\in\R_{>0} $.
Then,
\begin{equation*}
\begin{split}
|\tilde{\bphi}^\top G^\top \tilde\bxi|&\leq \frac{1}{2}(\|G\tilde{\bphi}\|^{2}+\|\tilde\bxi\|^2)\\
&\leq \frac{\max\{1,\lambda_{\rm max}(L)^2\}}{2}\|\col (\tilde{\bxi},\tilde{\bphi})\|^2,
\end{split}
\end{equation*}
where $\lambda_{\rm max}(L)$ is the largest eigenvalue of $L$. 
	Consequently, \eqref{iss_1} is obtained by considering $ \kappa\in (0,\kappa_1) $, {$ \alpha_1(\|\col (\tilde{\bxi},\tilde{\bphi})\|)=\alpha_1  \|\col (\tilde{\bxi},\tilde{\bphi})\|^2$, and $ \alpha_2(\|\col (\tilde{\bxi},\tilde{\bphi})\|)=\alpha_2  \|\col (\tilde{\bxi},\tilde{\bphi})\|^2$, with $ \kappa_1=\frac{1}{\max\{1,\lambda_{\rm max}(L)^2\}} $  and
\begin{align*}
\alpha_1&=\frac{1-\kappa\max\{1,\lambda_{\rm max}(L)^2\}}{2}\\
\alpha_2&=\frac{1+\kappa\max\{1,\lambda_{\rm max}(L)^2\}}{2}.
\end{align*}}
We compute the time derivative of $ V $ along the solutions of the system, and use \eqref{dist_dt_er2} together with $ \epsilon $-strong monotonicity of $ F(\bxi) $ to obtain
\begin{equation}\label{vdot_iss}
\begin{split}
\dot{V}\leq -\epsilon\|\bxi\|^2&+ \kappa \|G^\top \tilde\bxi\|^{2}-\kappa \tilde{\bphi}^\top G^\top (F(\bxi)-F(\bar\bxi))\\
&-\kappa \| G\tilde{\bphi}\|^{2}+(\tilde\bxi+\kappa G\tilde{\bphi})^\top \bnu.
\end{split}
\end{equation}
Define
\begin{equation*}
U(\bxi,\bar{\bxi}):=\int_{0}^{1}\nabla F(\bar{\bxi}+s(\bxi-\bar{\bxi}))ds.
\end{equation*}
By the fundamental theorem of calculus, we have
\begin{equation*}
F(\bxi)-F(\bar\bxi)=U(\bar{\bxi},\bxi)\tilde{\bxi}.
\end{equation*}
Substituting the equality above into \eqref{vdot_iss} yields 
\begin{equation}\label{vdot_d}
\dot{V}\leq -\col (\tilde{\bxi},G\tilde{\bphi})^\top P(\bxi, \bar \bxi) \,\,\col (\tilde{\bxi},G\tilde{\bphi})+\col (\tilde{\bxi},\tilde{\bphi})^\top R \,\bnu,
\end{equation}
where
\begin{equation*}
P(\bxi, \bar \bxi):=\begin{bmatrix}
\epsilon I-\kappa  GG^\top&\frac{1}{2}\kappa U(\bxi, \bar \bxi)^\top \\
\frac{1}{2}\kappa U(\bxi, \bar \bxi)& \kappa I
\end{bmatrix},\quad R:=\begin{bmatrix}
I\\ \kappa G^\top 
\end{bmatrix}.
\end{equation*}
Clearly, the matrix $ P $ is positive definite if and only if, for all $\bxi\in \R^{2nN}$,
\begin{equation*}
\kappa>0,\qquad \epsilon I-\kappa  GG^\top-\frac{1}{4}\kappa U(\bxi, \bar \bxi)^\top U(\bxi, \bar \bxi)>0.
\end{equation*}
{By using $ \|\nabla f_i(x_i, \sigma_i)\|\leq \gamma_i $, it is straightforward to investigate that $ \|U(\cdot, \cdot)\|^2\leq \bar\gamma^2  + \bar h\,^2+ 1$, where $\bar \gamma:=\max_{i\in \cali} (\gamma_i k_i)$ and  $\bar h:=\max_{i\in \cali} h_i$.}
Hence, we conclude that $ P>0 $ if $ \kappa\in (0,\kappa_2  )$ with $ \kappa_2=\frac{4\epsilon}{\bar\gamma^2  + \bar h\,^2+ 1+4\lambda_{\rm max}(L){^2}} $. Therefore, there exists $ \delta>0 $ such that $ P\geq \delta I $. Moreover, we have $ \|R\|=\sqrt{1+\kappa^2\lambda_{\rm max}(L){^2}} $. Then, by \eqref{vdot_d}, we find that
\begin{equation}\label{vdot_iss1}
\dot{V}\leq -\delta \|\col (\tilde{\bxi},G\tilde{\bphi})\|^2+\sqrt{1+\kappa^2\lambda_{\rm max}(L)^2}\|\col (\tilde{\bxi},\tilde{\bphi})\|\|\bnu\|.
\end{equation}
Noting that $ \tilde{\bphi}\in\ima(\bpi) $, we have
\begin{equation*}
\|G\tilde{\bphi}\|\geq  {\lambda_{\rm min}(L) \|\tilde{\bphi}\|},
\end{equation*}
where $ \lambda_{\rm min} (L) $ is the smallest nonzero eigenvalue of $ L $. This together with  \eqref{vdot_iss1} results in 
\begin{equation*}
\dot{V}\leq -m\|\col (\tilde{\bxi},\tilde{\bphi})\|^2+\sqrt{1+\kappa^2\lambda_{\rm max}(L)^2}\|\col (\tilde{\bxi},\tilde{\bphi})\|\|\bnu\|,
\end{equation*}
where $ m:=\delta  \min\{1,\lambda_{\rm min}(L)^2\} $. Hence, \eqref{iss_2} is obtained by setting {$ W(\tilde{\bxi},\tilde{\bphi})=\alpha_3\|\col (\tilde{\bxi},\tilde{\bphi})\|^2 $ and $ \rho(\|\bnu \|)=\alpha_4 \|\bnu\| $ with
\begin{align*}
\alpha_3&=m(1-\beta )\\
\alpha_4&=\frac{1}{\beta m}\sqrt{1+\kappa^2\lambda_{\rm max}(L)^2},
\end{align*}}
for some $\beta\in (0, 1)$. 
Consequently,  \eqref{dist_dt_er2} is ISS for $ 0<\kappa<\min\{\kappa_1,\kappa_2  \}$, {and we have
\begin{align*}
\beta_0(\|\col (\tilde{\bxi}(0),\tilde{\bpsi}(0))\|,t)&=\sqrt{\frac{\alpha_2}{\alpha_1}}e^{-\frac{\alpha_3}{2\alpha_2}t}\|\col (\tilde{\bxi}(0),\tilde{\bpsi}(0))\|\\
\beta_1 (\sup_{0\leq \tau \leq t}\|\bnu(\tau)\|)&=\sqrt{\frac{\alpha_2}{\alpha_1}}\alpha_4 \sup_{0\leq \tau \leq t}\|\bnu(\tau)\|.
\end{align*}}
 {
Note that we have shown the ISS property in the coordinates $(\tilde\bxi, \tilde\bphi)$, where $\tilde\bphi=\bpi\tilde\bpsi$. 
In addition, note that the incremental form \eqref{e:inc-dynamics} can be  written with respect to any equilibrium $(\bar\bxi , \bar\bpsi)$, where $\bar\bxi=\bxi^*$ and $\bar\bpsi \in \Psi$ by Proposition \ref{prop:eqil}. To conclude the ISS property of  \eqref{dist_dt} with respect to the equilibrium point $({\bxi}^*,{\bpsi}^*)$, it suffices to show that $\bpi ( \bpsi(t) - \bpsi^*)= \bpsi(t) - \bpsi^*$, for all $ t$.  The latter is equivalent to
\[
(\bone_N \bone_N^\top \otimes I_n) (\bpsi(t) - \bpsi^*)=\bze,
\]
which can be rewritten as
\[
\bone_N \otimes (\bone_N^\top \otimes I_n) (\bpsi(t) - \bpsi^*)=\bze.
\]
Noting that $(\bone_N^\top \otimes I_n) \bpsi(t) $ is a conserved quantity of the system, the above equality reduces to $(\bone_N^\top \otimes I_n) (\bpsi(0) - \bpsi^*)=\bze$. The latter holds, noting the definition of  $\bpsi^*$ in the theorem. This completes the proof.}
%
\end{proof}
%

%
\begin{remark}\hspace{-1mm}\footnote{The authors thank Sergio Grammatico for pointing out this connection.}
{In the case of general games and by considering suitable assumptions on  the pesudo-gradient mapping,  the presented NE seeking algorithm  in \cite{gadjov2018passivity} is exponentially stable \cite[Thm. 1 and 2]{gadjov2018passivity}. Therefore, it is also ISS with respect to additive time-varying disturbances \cite[Lem. 4.6]{khalil2002nonlinear}. However, that algorithm is fundamentally different than ours, which makes the analysis dissimilar. Specifically, the consensus term in \cite{gadjov2018passivity} appears as damping on the relative state variables, which contributes to  the exponential convergence property. For our presented algorithm, the consensus action appears as cross terms, resulting in the  presence of undamped communicating  variables $ \bpsi $ in \eqref{dist_dt}.  To overcome  this technical difficulty, we included a sufficiently small cross-term in the ISS Lyapunov function (see the second term in the right hand side of \eqref{ISS_Function}).}
\end{remark}	
	

{
\begin{remark}
The assumption of the boundedness of $\|\nabla f_i(\cdot,\cdot)\|$ can be relaxed at the expense of {establishing ISS in a local sense. In particular, for any compact set around the equilibrium, one can find restriction on the size of the disturbance such that ISS locally holds \cite{mironchenko2016local}}. 
\end{remark}}

\section{Distributed NE Seeking Dynamics for Constrained Actions}\label{projection}
This section extends the NE dynamics to the case when the action set is constrained to a compact set.  Let  the action set be given by $\calx_i \subset \R^n$, {and consider the following assumption.}  


\begin{assumption}[{{\cite[Asm. 2(ii)]{gadjov2018passivity}}}]\label{asmpset}
	For all $ i\in\cali $, the action set  $ \calx_i\subset\R^n $ is non-empty, convex, and compact, and the cost function $ J_i $ is $ \mathcal{C}^1 $ in all its arguments. 
\end{assumption}

The following lemma proves that the game has a unique NE.

\begin{lemma}\label{lemNEset}
	Let Assumptions~\ref{asmp1} and \ref{asmpset} be satisfied, 	then the aggregative game $ {\cal G}_{\text{agg}}=\big({\cal I} ,(J_i)_{i\in \cal I}, (\calx  _i)_{i\in \cal I}\big) $ with the cost function \eqref{cost} has a unique NE $ \bx^*\in\calx $  which is the solution of the variational inequality VI$ (\allowdisplaybreaks\calx,\bk\col\big((f_i(x_i,\avg{\bx}))_{i\in\cali}\big)) $ with $ \calx:=\prod_{i\in \cali}\calx_i $, $ f_i(\cdot) $ defined as \eqref{fi}, and $ k_i $ selected as \eqref{inter}. 
\end{lemma}
\begin{proof}
The claim can be proven by suitably adapting the results of  \cite{facchinei2007finite}. For the sake of completeness, we have provided a proof in Appendix A.
\end{proof}
To obtain the NE of the game in a distributed fashion, we again assume that the agents can exchange some variables through a connected undirected graph $ G_c $, and consider the following algorithm
\begin{equation}\label{distset_indv}
 \begin{aligned}
\dot{x}_i&=\Pi_{\calx_i}\left(x_i,-k_if_i(x_i,\sigma_i)\right)\\
\dot{{\sigma}}_i&=-{\sigma}_i+h_ix_i-\sum_{j\in\caln_i}(\psi_i-\psi_j)\\
\dot{\psi}_i&=\sum_{j\in\caln_i}({\sigma}_i-{\sigma}_j),
\end{aligned}
\end{equation}
where $ i\in\cali $ and $ \Pi_{\calx_i}(x_i,v) $ is the projection operator of the vector $ v\in\R^n $ on to the tangent cone of $ \calx_i $ at the point $ x_i\in\calx_i $. In vector form we have
\begin{equation}\label{distset}
\begin{split}
\dot{\bx}&=\Pi_{\calx}\left(\bx,-\bk\col\big((f_i(x_i,\sigma_i))_{i\in\cali}\big)\right)\\
\dot{\bsg}&=- \bsg+\bt\bx-(L\otimes I_n)\bpsi\\
\dot{\bpsi}&=(L\otimes I_n)\bsg.
\end{split}
\end{equation}
Note that \eqref{distset} is a discontinuous dynamical algorithm due to the projection operator. Therefore, we briefly discuss existence and uniqueness of solutions for this system.
Consider the collective projected-vector form of the algorithm as follows
\begin{equation*}
\col(\dot{\bx}, \dot{\bsg}, \dot{\bpsi})=\Pi_{\calx\times \R^{nN}\times \R^{nN}}\Big(\col({\bx}, {\bsg}, {\bpsi}),-F_{\text{ext}(\bx,\bsg,\bpsi)}\Big),
\end{equation*}
where 
\begin{equation*}
F_{\text{ext}(\bx,\bsg,\bpsi)}:=\left[\begin{array}{c}
F(\bx,\bsg)+G\bpsi\\
-(L\otimes I_n)\bsg
\end{array}\right],
\end{equation*}
with $ F(\bx,\bsg) $ given by \eqref{F} and $ G:=\col(\bze, (L\otimes I_n)) $. 
 Using Assumption~\ref{asmpset} and the fact that $\R^{nN}  $ is a clopen set (closed-open set), the set $ \calx\times \R^{nN}\times \R^{nN} $ is closed and convex. In addition, by considering Lemma~\ref{pro1}, it follows that $ F_{\text{ext}(\bx,\bsg,\bpsi)} $ is monotone. Therefore, from \cite[Thm. 1]{brogliato2006equivalence}, we conclude that for any initial condition  $ ({\bx}(0),{\bsg}(0),\bpsi(0) )\in \calx \times \R^{nN}\times \R^{nN}$, the algorithm \eqref{distset} has a unique solution which belongs to $ \calx \times \R^{nN}\times \R^{nN} $ for almost all $ t\geq 0 $. Converging the algorithm to a point corresponding to the NE of the game is established next.
\begin{theorem}\label{the:convSet}
	{Let Assumptions~\ref{asmp1} and \ref{asmpset} be satisfied, and }consider the NE seeking algorithm \eqref{distset} with  initial condition $ ({\bx}(0),{\bsg}(0),\bpsi(0) )\in \calx\times \R^{nN}\times \R^{nN}$. Then, the solution $ ({\bx},{\bsg},\bpsi) $   converges to the equilibrium point $ (\bar\bx,\bar\bsg,\bar\bpsi)=(\bx^*,\bone_N \otimes \avg{\bx^*}, \bpsi^*) $ where $ \bx^* $  is the unique NE of the aggregative game $ \mathcal{G}_{\text{agg}} $ and  $\bpsi^*\in \Psi$ is given by $\bpsi^*=(L^{+}\otimes I_n) \bt\bx^*+\frac{1}{N} (\bone_N\bone_N^\top \otimes I_n) \bpsi(0)$.
\end{theorem}
\begin{proof}
At the equilibrium point $ (\bar\bx,\bar\bsg,\bar\bpsi) $, by  \eqref{distset}, we have
	\begin{equation}\label{equilSet}
	\begin{split}
	\bze&=\Pi_{\calx}\left(\bar\bx,-\bk\col\big((f_i(\bar x_i,\bar\sigma_i))_{i\in\cali}\big)\right)\\
	\bze&=-\bar \bsg+\bt\bar \bx-(L\otimes I_n)\bar \bpsi\\
	\bze&=(L\otimes I_n)\bar\bsg.
	\end{split}
	\end{equation}
	Similar to the proof of Proposition~\ref{prop:eqil}, it can be shown that $\bar \bsg=\bone_N\otimes \avg{\bar\bx}  $.	Consequently, we obtain the following equality by using  Moreau's decomposition theorem
	\begin{equation*}
	\begin{split}
	\bze&=\Pi_{\calx}\left(\bar\bx,-\bk\col\big((f_i(\bar x_i,\avg{\bar \bx}))_{i\in\cali}\big)\right)\\
	&=-\bk\col\big((f_i(\bar x_i,\avg{\bar \bx}))_{i\in\cali}\big)\\
	&\quad -\text{proj}_{\caln_{\calx}(\bar{\bx})}\left(-\bk\col\big((f_i(\bar x_i,\avg{\bar \bx}))_{i\in\cali}\big)\right),
	\end{split}
	\end{equation*}
	where $ \caln_{\calx}(\bar{\bx}) $ is the normal cone of $ \calx $ at $ \bar\bx\in\calx $. This means that 
	\begin{equation*}
	-\bk\col\big((f_i(\bar x_i,\avg{\bar \bx}))_{i\in\cali}\big)\in \caln_{\calx}(\bar\bx).
	\end{equation*}
	In other words, $ \bar{\bx} $ is the solution of VI$ (\calx,\bk\col\big((f_i( x_i,\avg{ \bx}))_{i\in\cali}\big)) $, and from Lemma~\ref{lemNEset}, we conclude that $ \bar{\bx}=\bx^* $. The proof of $ \bar \bpsi\in\varPsi $ is similar to Proposition~\ref{prop:eqil}.
	
	To show convergence, let $ \bxi:=\col({\bx},{\bsg}) $, $\Lambda:=\calx\times \R^{nN}  $, $ F(\bxi)=F(\bx,{\bsg}) $, and the Lyapunov function candidate $ V(\tilde\bxi,\tilde\bpsi):=\frac{1}{2}\|\col(\tilde\bxi,\tilde\bpsi)\|^2 $ with $ \tilde\bxi=\bxi-\bar \bxi $ and $ \tilde{\bpsi}=\bpsi-\bar \bpsi $. By using \eqref{distset} and the definition of $ G $, we obtain
	\begin{equation*}
	\dot{V}=\tilde\bxi^\top \Pi_{\Lambda}\big(\bxi,-F(\bxi)-G\bpsi\big)+\tilde{\bpsi}^\top G^\top\bxi.
	\end{equation*}
	By  Moreau's decomposition theorem, we find that
	\begin{equation*}
	\begin{split}
	\tilde\bxi^\top \Pi_{\Lambda}\big(\bxi,-F(\bxi)-G\bpsi\big)&=\tilde\bxi^\top\Big(-F(\bxi)-G\bpsi\\
	&-\text{proj}_{\caln_{\Lambda}(\bxi)}\big(-F(\bxi)-G\bpsi\big)\Big).
	\end{split}
	\end{equation*}
	Noting $ \bar\bxi\in \Lambda $, we have
	\begin{equation*}
	-\tilde\bxi^\top \text{proj}_{\caln_{\Lambda}(\bxi)}\big(-F(\bxi)-G\bpsi\big)\leq 0,
	\end{equation*}
	and the time derivative of $ V $ admits the following inequality
	\begin{equation}\label{dot_v_set}
	\dot{V}\leq -\tilde\bxi^\top F(\bxi)- \tilde\bxi^\top G\bpsi+\tilde{\bpsi}^\top G^\top\bxi.
	\end{equation}
Moreover, from \eqref{equilSet} and Moreau's decomposition theorem we get
	\begin{equation*}
	\begin{split}
	0&=\tilde{\bxi}^\top \Pi_{\Lambda}\big(\bar \bxi,-F(\bar \bxi)-G\bar \bpsi\big)\\
	&=\tilde{\bxi}^\top\Big(-F(\bar\bxi)-G\bar \bpsi-\text{proj}_{\caln_{\Lambda}(\bar\bxi)}\big(-F(\bar\bxi)-G\bar \bpsi\big)\Big).
	\end{split}
	\end{equation*}
	Since
	\begin{equation*}
	-\tilde{\bxi}^\top\text{proj}_{\caln_{\Lambda}(\bar\bxi)}\big(-F(\bar\bxi)-G\bar \bpsi\big)\geq 0,
	\end{equation*}
	we conclude that $
	\tilde{\bxi}^\top\Big(F(\bar\bxi)+G\bar \bpsi\Big)\geq 0
	$, which can be employed to rewrite \eqref{dot_v_set} as
		\begin{equation*}
	\begin{split}
	\dot{V}&\leq -\tilde{\bxi}^\top \left(F(\bxi)-F(\bar\bxi)\right)- \tilde\bxi^\top G\tilde{\bpsi}+\tilde{\bpsi}^\top G^\top\bxi\\
	&\leq -\epsilon \|\tilde{\bxi}\|^2,
	\end{split}
	\end{equation*}
	where the last inequality is obtained by using \eqref{equilSet} and the fact that $ F(\bxi) $ is $ \epsilon $-strongly monotone. By following an analogous argument to \cite[Thm. 2]{de2018distributed}, we conclude   that the solution $ ({\bx},{\bsg},\bpsi) $ converges to  the set 
$
\Omega=\left\{({\bx},{\bsg},\bpsi ) \mid \bx=\bar{\bx},\,  \bsg=\bar{\bsg}, {\bm{\psi}\in \Psi} \right \}
$. 
{Noting that $(\bone_N^\top \otimes I_n)\bpsi(t)$ is an invariant quantity of the system, similar to Theorem \ref{th:conv}, we conclude that $\bpsi$ converges to $\bpsi^*$.}  
\end{proof}
\section{Case studies}\label{simulation}
In this section, we consider two illustrative case studies that are formulated as aggregative games. 

\subsection{Energy Consumption Game}
This case study considers the energy consumption problem of consumers with heating ventilation air conditioning (HVAC) systems in smart grids. As proposed in \cite{ma2014distributed}, this problem can be formulated into a noncooperative game where each consumer chooses its energy consumption such that the following payoff function is minimized
\begin{equation*}
J_i(x_i,\avg{x})=\theta \gamma^2 (x_i-\hat{x}_i)^2+(aN\avg{\bx}+b)x_i,
\end{equation*}
where the positive constant parameters $ \theta $, $ \gamma $, and $ a $ are the cost, the thermal, and the price-elasticity coefficients, respectively. The scalar $ b\in\R_{>0} $ is a basic price for unite energy consumption, $ x_i\in \calx_i $ is the energy consumption of consumer $ i $,  $ \hat{x}_i\in \calx_i $ is the required energy consumption for maintaining the target indoor temperature, and $ N\avg{\bx}=\sum_{j\in\cali}x_j $ is the total energy consumption.  The action set $ \calx_i\subset \R $  is defined as $ \calx_i:=\big\{x_i\in\R\mid x_i\in[\underline{x}_i,\bar x_i]\big\} $ where the positive constants $ \underline{x}_i $ and $ \bar x_i $ are the minimum and maximum acceptable energy consumption, respectively, with $ \underline{x}_i<\bar x_i $. According to \cite[Thm. 1]{ma2014distributed}, this game has a unique NE if 
\begin{equation*}
a\leq 2\theta \gamma^2/(N-3) ,
\end{equation*}
for $ N>3 $. 
If we use Lemma~\ref{lemNEset}, the sufficient condition for having a unique NE is $ a\leq 2\theta \gamma^2/(N-1) $ for $ N>1 $, which is slightly more restrictive than the above condition. 
However, considering Remark~\ref{r:relax}, we need to find $ k_i>0 $ such that the mapping $ \col(k_i f_i(x_i,\sigma_i),\sigma_i-x_i) $ with $ f_i(x_i,\sigma_i)=(2\theta \gamma^2+a)x_i+aN \sigma_i-2\theta \gamma^2 \hat{x}_i +b$ is strongly monotone. By performing the calculations, we obtain that for all $ a\in \R_{>0} $ and $ N\geq 1 $, the mapping is strongly monotone if 
\begin{multline*}
k_i\in \big(\frac{(\sqrt{2\theta \gamma^2+a}-\sqrt{2\theta \gamma^2+a(N+1)})^2}{(aN)^2},\\
\frac{(\sqrt{2\theta \gamma^2+a}+\sqrt{2\theta \gamma^2+a(N+1)})^2}{(aN)^2}\big),
\end{multline*}
which means that we need less restrictive assumptions to guarantee uniqueness of the NE and convergence of the algorithm. We consider $ N=5 $ players in this game, i.e., $ \cali=\{1,\cdots,5\} $, with $ \theta \gamma^2 $ normalized to one, $ \col((\hat x_i )_{i\in \cali})=\col(50,55,60,65) (\text{kWh})$, $ \col((\bar x_i)_{i\in \cali})=\col(60,66,72,78,84) (\text{kWh})$, $ \col((\underline x_i)_{i\in \cali})=\col(40,44,46,52,56) (\text{kWh})$, $ a=0.04 $, and $ b=5 (\text{\textdollar}/(\text{kWh}))$ \cite{ye2016game}. To implement the algorithm, the players are assumed to communicate through a connected undirected graph depicted in Fig.~\ref{fig:graph_HVAC}. Each player randomly chooses the design parameter $ k_i $  in the above interval. The initial conditions of $ \sigma_i $ and $ \psi_i $ are chosen randomly, and $ x_i(t_0)=0.5(\bar x_i+\underline x_i) $. The resulting action variables are depicted in Fig.~\ref{fig:results_HVAC}. The fact that the players converge to the NE of the game can be verified by comparing the results to the NE computed in \cite[Sec. VI-C]{ye2016game}.

Next, we consider the case where {the action set is $ \calx_i=\R $ and} bounded disturbances affect the dynamics, and investigate its robustness. The disturbance vector $ \bnu\in \R^{10} $ is added according to \eqref{dist_dt}, five elements of which are considered as uniformly distributed random numbers in the interval $ [-20,20] $ with the sampling time $ 0.1 $(s), and the other five elements are sinusoidal signals with  amplitudes between 10 and 20, and frequencies between 5 to 25(rad/s). As can be seen from Fig.~\ref{fig:results_ISS_HVAC}, the action variables remain bounded, which is consistent with our ISS results. Note that the presence of disturbances results in  deviation of the asymptotic behavior from the NE. 
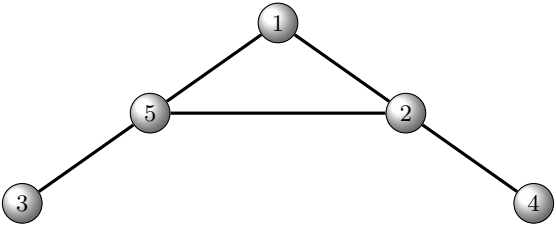
\begin{figure}
	\[\begin{tikzpicture}[x=1.7cm, y=1.2cm,
	every edge/.style={sloped, draw, line width=1.2pt}]
	
	\vertex (v1) at (0,0)  {\small $1$};
	\vertex (v2) at (1,-1) {\small $2$};
	\vertex (v5) at (-1,-1) {\small $5$};
	\vertex (v4) at (2,-2) {\small $4$};
	\vertex (v3) at (-2,-2) {\small $3$};
	%
	\path
	(v1) edge  (v2)
	(v1) edge (v5)
	(v2) edge (v5)
	(v2) edge (v4)
	(v5) edge (v3);
	\end{tikzpicture}
	\]
	\caption{Communication Graph in HVAC Example.}\label{fig:graph_HVAC}
\end{figure}

\begin{figure}
	\centering
	\includegraphics[width=3in]{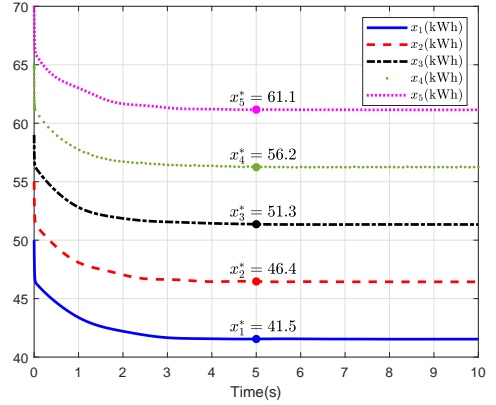}
	\caption{Action variables of consumers with HVAC systems.}
	\label{fig:results_HVAC}
\end{figure}

\begin{figure}
	\centering
	\includegraphics[width=3in]{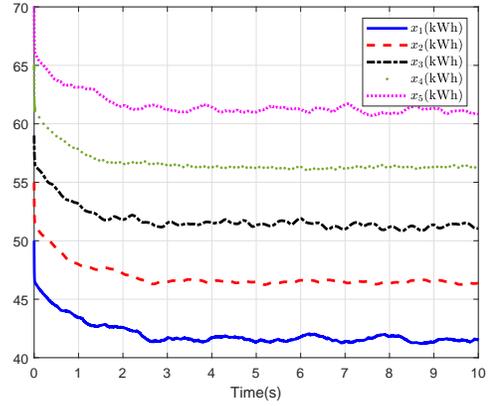}
	\caption{Action variables of consumers with HVAC systems in the presence of disturbances.}
	\label{fig:results_ISS_HVAC}
\end{figure}

\subsection{Charging Coordination of Electric Vehicles}
In this subsection, we consider the problem of charging coordination for a population $ \cali=\{1,\cdots,N\} $ of plug-in electrical vehicles (PEVs) \cite{ma2011decentralized,ma2015distributed}.  Each agent is aimed at minimizing its payoff function defined as the summation of its electricity bill and a quadratic function as follows 
\begin{equation*}
J_i(x_i,\avg{\bx})=\sum_{t\in\calt}\big(a(d_t+N\avg{\bx^t})+b\big) x_i^t+q_i(x_i^t)^2+c_i x_i^t,
\end{equation*}
where $ \calt:=\{1,\cdots,n\} $ is the charging horizon, $ x_i=\col\big((x_i^t)_{t\in\calt}\big) $ is the collection of charging control of the $ i $-th vehicle at time $ t $, the positive constants $ a $ and $ b $ respectively are the price-elasticity coefficient and basic price, $ d_t $ is the total non-PEV demand,  $ N\avg{\bx^t}=\sum_{j\in\cali}x_j^t $ is the total PEV demand at time $ t $, and $ q_i $ and $ c_i $ are positive constant parameters. In the payoff function, the quadratic term models battery degradation cost of PEVs \cite{ma2015distributed}. For each agent, the charging rate $ x_i^t $ is bounded as $ 0\leq x_i^t\leq \bar x_i $ and its summation for all $ t\in\calt $ should be equal to the required energy of the agent defined as $ \gamma_i $. Therefore, the constraint set of $ x_i $ is $ \calx_i:=\calx_i^1\cap\calx_i^2 $ where
\begin{equation}\label{pev_sets}
\begin{split}
\calx_i^1:&=\big\{x_i\in\R^{n}\mid x_i^t\in[0,\bar x_i]\big\}\\
\calx_i^2:&=\big\{x_i\in\R^{n}\mid  \sum_{t\in\calt}x_i^t=\gamma_i\big\}.
\end{split}
\end{equation} 
In practice, it is assumed that $ n\bar x_i\geq \gamma_i $ to grantee that $ \calx_i $ is non-empty. The goal is to reach to the NE and schedule charging strategies for the entire horizon, and in this regard, a gather and broadcast algorithm is presented in \cite{ma2015distributed} which guarantees convergence when $ q_i>aN $ \cite[Thm. 3.1]{ma2015distributed}. 

In this problem, we have $ f_i(x_i,\sigma_i)=(2q_i+a)x_i+aN\sigma_i+ad+(b+c_i)\bone_n $ with $ d=\col\big((d_t)_{t\in\calt}\big) $; therefore, the mapping $ \col(k_i f_i(x_i,\sigma_i),\sigma_i-x_i) $ is strongly monotone by choosing
\begin{multline*}
k_i\in \big(\frac{(\sqrt{2q_i+a}-\sqrt{2q_i+a(N+1)})^2}{(aN)^2},\\
\frac{(\sqrt{2q_i+a}+\sqrt{2q_i+a(N+1)})^2}{(aN)^2}\big),
\end{multline*}
and there is no need for the assumption $ q_i>aN $. To reach  the NE, each agent can implement \eqref{distset_indv}; however, since $ \calx_i $ is the intersection of two sets, it is not easy to find a closed-form expression for the projection operator $ \Pi_{\calx_i}\left(x_i,\cdot\right) $. To overcome this challenge, {we use the fact that 
$ x_i(t) $ in the NE dynamics does not need to belong to $ \calx_i $ for all $ t\geq t_0 $, yet it should converge to the NE inside this set. Therefore, $x_i\in\calx_i^2$ can be treated as a ``soft constraint". Hence, we modify  \eqref{distset}  as follows}
\begin{equation}\label{ne_alg_lagrange}
\begin{split}
\dot{\bx}&=\Pi_{\calx^1}\left(\bx,-\bk\col\big((f_i(x_i,\sigma_i))_{i\in\cali}\big)-(I_N\otimes \bone_n)\bm{\lambda}\right)\\
\dot{\bsg}&=- \bsg+\bt\bx-(L\otimes I_n)\bpsi\\
\dot{\bpsi}&=(L\otimes I_n)\bsg\\
\dot{\bm{\lambda}}&=(I_N\otimes \bone_n^\top)\bx-\bm \gamma,
\end{split}
\end{equation} 
where $ \bm{\lambda}=\col((\lambda_i)_{i\in\cali}) $ with the Lagrangian multiplier $ \lambda_i\in \R $, $\bm\gamma=\col((\gamma_i)_{i\in\cali})$,  and $ \calx^1=\prod_{i\in \cali}\calx_i^1 $ with $ \calx_i^1 $  defined in \eqref{pev_sets}. 
A supplementary discussion on the convergence of the above algorithm to the NE is provided in  Appendix B.

A population of $ N=100 $ players, that can communicate by a connected undirected graph, are considered in this game, and the charging horizon is from 12:00 a.m. on one day to 12:00 a.m. on the next day. In order to generate the numerical parameters, we consider some nominal values and randomize them similar to \cite{grammatico2017dynamic}. In the price function, $ a=3.8\times 10^{-3} $ and $ b=0.06(\text{\textdollar}/(\text{kWh})) $ are considered. The parameters of the quadratic functions are uniformly distributed random numbers as $ q_i \sim \{0.004\}+[-0.001,\, 0.001]$ and $ c_i \sim \{0.075\}+[-0.01,\, 0.01]$. In order to generate $ \gamma_i $, inspired by \cite{ma2015distributed}, we assume that the battery capacity size of PEVs are $ \Phi_i\sim \{30\}+[-5,\, 5](\text{kWh}) $, the initial  states of charge ($\text{SOC}_{i_0}$) of PEVs satisfy a Gaussian distribution with the mean 0.5 and variance 0.1, and the  final  state of charge ($\text{SOC}_{i_f}$) equals to 0.95; thus, $ \gamma_i=\Phi_i(\text{SOC}_{i_f}-\text{SOC}_{i_0}) $. In addition, the maximum admissible charging control is set to  $ \bar x_i \sim \{10\}+[-2,\, 2](\text{kWh})$.

We select the design parameter of the algorithm as $ k_i=(2(2q_i+a)+aN)/(aN)^2 $, the initial condition of action variables are $ x_i^t(t_0)=\gamma_i/n $, and $ \sigma_i(t_0) $, $ \psi_i(t_0) $, and $ \lambda_i(t_0) $ are selected randomly. Fig.~\ref{fig:results_PEV} illustrates total demand and total non-PEV demand. As can be seen, the PEVs shifted their charging intervals to the nighttime, which minimizes their effects on the grid, and as explained in \cite{ma2011decentralized}, the NE has the desired ``valley filling'' property.

\begin{figure}
	\centering
	\includegraphics[width=3in]{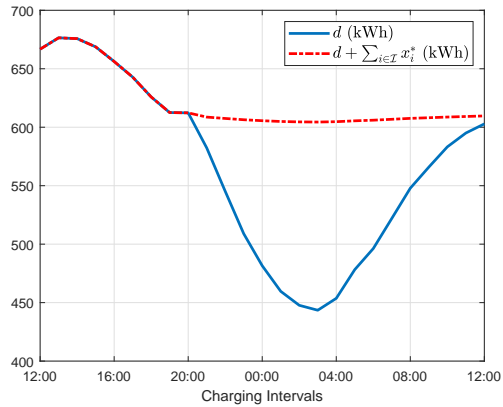}
	\caption{Total  non-PEV demand $ d $ and its summation with total-PEV demand at the equilibrium $ d+\sum_{i\in\cali}x_i^* $.}
	\label{fig:results_PEV}
\end{figure}

\section{Conclusions}\label{conclusion}
{By employing the structure of aggregative games, we presented a distributed NE seeking algorithm where each player calculates its action variable through computing an estimation of the aggregation term. We provided sufficient  conditions for  convergence of the  algorithm  to the NE of the game. 
We have provided 	
privacy guarantees for the algorithm by showing that private information of the players cannot be reconstructed even if all communicated variables are accessed by an adversary. {Raised by practical concerns about the accuracy of payoff functions or rationality of the players, we proved robustness of the proposed algorithm against time-varying disturbances in the sense of ISS.} 
Finally, we extended the algorithm to the case of constrained action sets  by using projection operators. Extension of the results to games with coupling constraints is left for future work. Moreover, a challenging task is to provide robustness guarantees in the presence of projections. 
Another notable research questions is to use aggregative game dynamics as optimal controllers steering a physical system.  Examples of the latter in Cournot and Bertrand competitions can be found in \cite{de2019feedback} and  \cite{stegink2018hybrid}, respectively.}
\bibliographystyle{IEEEtran}
\bibliography{MyReferences}

\begin{thebibliography}{10}
\providecommand{\url}[1]{#1}
\csname url@samestyle\endcsname
\providecommand{\newblock}{\relax}
\providecommand{\bibinfo}[2]{#2}
\providecommand{\BIBentrySTDinterwordspacing}{\spaceskip=0pt\relax}
\providecommand{\BIBentryALTinterwordstretchfactor}{4}
\providecommand{\BIBentryALTinterwordspacing}{\spaceskip=\fontdimen2\font plus
\BIBentryALTinterwordstretchfactor\fontdimen3\font minus
  \fontdimen4\font\relax}
\providecommand{\BIBforeignlanguage}[2]{{%
\expandafter\ifx\csname l@#1\endcsname\relax
\typeout{** WARNING: IEEEtran.bst: No hyphenation pattern has been}%
\typeout{** loaded for the language `#1'. Using the pattern for}%
\typeout{** the default language instead.}%
\else
\language=\csname l@#1\endcsname
\fi
#2}}
\providecommand{\BIBdecl}{\relax}
\BIBdecl

\bibitem{jensen2010aggregative}
M.~K. Jensen, ``Aggregative games and best-reply potentials,'' \emph{Economic
  theory}, vol.~43, no.~1, pp. 45--66, 2010.

\bibitem{mas1995microeconomic}
A.~Mas-Colell, M.~D. Whinston, and J.~R. Green, \emph{Microeconomic
  theory}.\hskip 1em plus 0.5em minus 0.4em\relax Oxford university press New
  York, 1995, vol.~1.

\bibitem{de2019feedback}
\BIBentryALTinterwordspacing
C.~De~Persis and N.~Monshizadeh, ``A feedback control algorithm to steer
  networks to a {Cournot}-{Nash} equilibrium,'' \emph{IEEE Transactions on
  Control of Network Systems}, pp. 1--12, 2019. [Online]. Available:
  \url{http://dx.doi.org/10.1109/TCNS.2019.2897907}
\BIBentrySTDinterwordspacing

\bibitem{mohsenian2010autonomous}
A.~H. Mohsenian-Rad, V.~W. Wong, J.~Jatskevich, R.~Schober, and A.~Leon-Garcia,
  ``Autonomous demand-side management based on game-theoretic energy
  consumption scheduling for the future smart grid,'' \emph{IEEE Transactions
  on Smart Grid}, vol.~1, no.~3, pp. 320--331, 2010.

\bibitem{ma2011decentralized}
Z.~Ma, D.~S. Callaway, and I.~A. Hiskens, ``Decentralized charging control of
  large populations of plug-in electric vehicles,'' \emph{IEEE Transactions on
  Control Systems Technology}, vol.~21, no.~1, pp. 67--78, 2011.

\bibitem{alpcan2005distributed}
T.~Alpcan and T.~Ba{\c{s}}ar, ``Distributed algorithms for {Nash} equilibria of
  flow control games,'' in \emph{Advances in dynamic games}.\hskip 1em plus
  0.5em minus 0.4em\relax Springer, 2005, pp. 473--498.

\bibitem{basar1999dynamic}
T.~Basar and G.~J. Olsder, \emph{Dynamic noncooperative game theory}.\hskip 1em
  plus 0.5em minus 0.4em\relax Siam, 1999, vol.~23.

\bibitem{facchinei2007finite}
F.~Facchinei and J.~S. Pang, \emph{Finite-dimensional variational inequalities
  and complementarity problems}.\hskip 1em plus 0.5em minus 0.4em\relax
  Springer Science \& Business Media, 2007.

\bibitem{salehisadaghiani2016distributed}
F.~Salehisadaghiani and L.~Pavel, ``Distributed {Nash} equilibrium seeking: A
  gossip-based algorithm,'' \emph{Automatica}, vol.~72, pp. 209--216, 2016.

\bibitem{ye2017distributed}
M.~Ye and G.~Hu, ``Distributed {Nash} equilibrium seeking by a consensus based
  approach,'' \emph{IEEE Transactions on Automatic Control}, vol.~62, no.~9,
  pp. 4811--4818, 2017.

\bibitem{gadjov2018passivity}
D.~Gadjov and L.~Pavel, ``A passivity-based approach to {Nash} equilibrium
  seeking over networks,'' \emph{IEEE Transactions on Automatic Control},
  vol.~64, no.~3, pp. 1077--1092, 2019.

\bibitem{de2018continuous}
\BIBentryALTinterwordspacing
C.~{De Persis} and S.~{Grammatico}, ``Continuous-time integral dynamics for a
  class of aggregative games with coupling constraints,'' \emph{IEEE
  Transactions on Automatic Control}, pp. 1--6, 2019. [Online]. Available:
  \url{http://dx.doi.org/10.1109/TAC.2019.2939639}
\BIBentrySTDinterwordspacing

\bibitem{grammatico2017dynamic}
S.~Grammatico, ``Dynamic control of agents playing aggregative games with
  coupling constraints,'' \emph{IEEE Transactions on Automatic Control},
  vol.~62, no.~9, pp. 4537--4548, 2017.

\bibitem{paccagnan2018nash}
D.~Paccagnan, B.~Gentile, F.~Parise, M.~Kamgarpour, and J.~Lygeros, ``{Nash}
  and {Wardrop} equilibria in aggregative games with coupling constraints,''
  \emph{IEEE Transactions on Automatic Control}, vol.~64, no.~4, pp.
  1373--1388, 2018.

\bibitem{koshal2016distributed}
J.~Koshal, A.~Nedi{\'c}, and U.~V. Shanbhag, ``Distributed algorithms for
  aggregative games on graphs,'' \emph{Operations Research}, vol.~64, no.~3,
  pp. 680--704, 2016.

\bibitem{parise2017distributed}
\BIBentryALTinterwordspacing
F.~{Parise}, B.~{Gentile}, and J.~{Lygeros}, ``A distributed algorithm for
  almost-{Nash} equilibria of average aggregative games with coupling
  constraints,'' \emph{IEEE Transactions on Control of Network Systems}, pp.
  1--12, 2019. [Online]. Available:
  \url{http://dx.doi.org/10.1109/TCNS.2019.2944300}
\BIBentrySTDinterwordspacing

\bibitem{Gadjov2019single}
D.~{Gadjov} and L.~{Pavel}, ``{Single-timescale distributed {GNE} seeking for
  aggregative games over networks via forward-backward operator splitting},''
  \emph{arXiv e-prints}, p. arXiv:1908.00107, Jul 2019.

\bibitem{liang2017distributed}
S.~Liang, P.~Yi, and Y.~Hong, ``Distributed {Nash} equilibrium seeking for
  aggregative games with coupled constraints,'' \emph{Automatica}, vol.~85, pp.
  179--185, 2017.

\bibitem{lei2018linearly}
J.~Lei and U.~V. Shanbhag, ``Linearly convergent variable sample-size schemes
  for stochastic {Nash} games: Best-response schemes and distributed
  gradient-response schemes,'' in \emph{IEEE Conference on Decision and Control
  (CDC)}.\hskip 1em plus 0.5em minus 0.4em\relax IEEE, 2018, pp. 3547--3552.

\bibitem{ye2016game}
M.~Ye and G.~Hu, ``Game design and analysis for price-based demand response: An
  aggregate game approach,'' \emph{IEEE Transactions on Cybernetics}, vol.~47,
  no.~3, pp. 720--730, 2016.

\bibitem{Monshizadeh2019plausible}
N.~Monshizadeh and P.~Tabuada, ``Plausible deniability as a notion of
  privacy,'' {T}o be presented at \textit{58th IEEE Conference on Decision and
  Control (CDC), 2019}. Preprint available at the homepage of the author
  \url{https://sites.google.com/site/nmonshizadeh/home}.

\bibitem{dwork2011differential}
C.~Dwork, ``Differential privacy,'' \emph{Encyclopedia of Cryptography and
  Security}, pp. 338--340, 2011.

\bibitem{dwork2014algorithmic}
C.~Dwork and A.~Roth, ``The algorithmic foundations of differential privacy,''
  \emph{Foundations and Trends{\textregistered} in Theoretical Computer
  Science}, vol.~9, no. 3--4, pp. 211--407, 2014.

\bibitem{Cortes2016}
J.~Cort{\'e}s, G.~E. Dullerud, S.~Han, J.~Le~Ny, S.~Mitra, and G.~J. Pappas,
  ``Differential privacy in control and network systems,'' in \emph{55th IEEE
  Conference on Decision and Control (CDC)}.\hskip 1em plus 0.5em minus
  0.4em\relax IEEE, 2016, pp. 4252--4272.

\bibitem{frihauf2011nash}
P.~Frihauf, M.~Krstic, and T.~Basar, ``{Nash} equilibrium seeking in
  noncooperative games,'' \emph{IEEE Transactions on Automatic Control},
  vol.~57, no.~5, pp. 1192--1207, 2011.

\bibitem{pita2010robust}
J.~Pita, M.~Jain, M.~Tambe, F.~Ord{\'o}{\~n}ez, and S.~Kraus, ``Robust
  solutions to {Stackelberg} games: Addressing bounded rationality and limited
  observations in human cognition,'' \emph{Artificial Intelligence}, vol. 174,
  no.~15, pp. 1142--1171, 2010.

\bibitem{fan2006passivity}
X.~Fan, T.~Alpcan, M.~Arcak, T.~Wen, and T.~Ba{\c{s}}ar, ``A passivity approach
  to game-theoretic {CDMA} power control,'' \emph{Automatica}, vol.~42, no.~11,
  pp. 1837--1847, 2006.

\bibitem{ito2012disturbance}
H.~Ito, ``Disturbance and delay robustness guarantees of gradient systems based
  on static noncooperative games with an application to feedback control for
  {PEV} charging load allocation,'' \emph{IEEE Transactions on Control Systems
  Technology}, vol.~21, no.~4, pp. 1374--1385, 2012.

\bibitem{cherukuri2017role}
A.~Cherukuri, E.~Mallada, S.~Low, and J.~Cort{\'e}s, ``The role of convexity in
  saddle-point dynamics: {Lyapunov} function and robustness,'' \emph{IEEE
  Transactions on Automatic Control}, vol.~63, no.~8, pp. 2449--2464, 2017.

\bibitem{weitenberg2018exponential}
E.~Weitenberg, C.~De~Persis, and N.~Monshizadeh, ``Exponential convergence
  under distributed averaging integral frequency control,'' \emph{Automatica},
  vol.~98, pp. 103--113, 2018.

\bibitem{sontag2008input}
E.~D. Sontag, ``Input to state stability: Basic concepts and results,'' in
  \emph{Nonlinear and optimal control theory}.\hskip 1em plus 0.5em minus
  0.4em\relax Springer, 2008, pp. 163--220.

\bibitem{khalil2002nonlinear}
H.~Khalil, \emph{Nonlinear Systems}, ser. Pearson Education.\hskip 1em plus
  0.5em minus 0.4em\relax Prentice Hall, 2002.

\bibitem{mironchenko2016local}
A.~Mironchenko, ``Local input-to-state stability: Characterizations and
  counterexamples,'' \emph{Systems \& Control Letters}, vol.~87, pp. 23--28,
  2016.

\bibitem{brogliato2006equivalence}
B.~Brogliato, A.~Daniilidis, C.~Lemar{\'e}chal, and V.~Acary, ``On the
  equivalence between complementarity systems, projected systems and
  differential inclusions,'' \emph{Systems \& Control Letters}, vol.~55, no.~1,
  pp. 45--51, 2006.

\bibitem{de2018distributed}
C.~De~Persis and S.~Grammatico, ``Distributed averaging integral {Nash}
  equilibrium seeking on networks,'' \emph{Automatica}, vol. 110, p. 108548,
  2019.

\bibitem{ma2014distributed}
K.~Ma, G.~Hu, and C.~J. Spanos, ``Distributed energy consumption control via
  real-time pricing feedback in smart grid,'' \emph{IEEE Transactions on
  Control Systems Technology}, vol.~22, no.~5, pp. 1907--1914, 2014.

\bibitem{ma2015distributed}
Z.~Ma, S.~Zou, and X.~Liu, ``A distributed charging coordination for
  large-scale plug-in electric vehicles considering battery degradation cost,''
  \emph{IEEE Transactions on Control Systems Technology}, vol.~23, no.~5, pp.
  2044--2052, 2015.

\bibitem{stegink2018hybrid}
T.~Stegink, A.~Cherukuri, C.~De~Persis, A.~Van~der Schaft, and J.~Cort{\'e}s,
  ``Hybrid interconnection of iterative bidding and power network dynamics for
  frequency regulation and optimal dispatch,'' \emph{IEEE Transactions on
  Control of Network Systems}, vol.~6, no.~2, pp. 572--585, 2018.

\bibitem{baker2016strong}
J.~Baker, ``Strong convexity does not imply radial unboundedness,'' \emph{The
  American Mathematical Monthly}, vol. 123, no.~2, pp. 185--188, 2016.

\end{thebibliography}

\appendices
\section*{Appendix A: Proofs of the Lemmas}
\noindent
\textit{Proof of Lemma~\ref{pro1}.} 	
$ (i) $ The mapping $ F $ is $ \epsilon $-strongly monotone if 
	\begin{multline}\label{ep_mono}
		\col\big(\bx-{\bx}',\bsg-{\bsg}'\big)^\top 
		\big(F(\bx, \bsg)- F(\bx', \bsg')\big)\geq \\
		\epsilon \|\bx-{\bx}'\|^2	+\epsilon \|\bsg-{\bsg}'\|^2
	\end{multline}
	for all $ \bx,\bx'\in \calx=\prod_{i\in \cali}\calx_i $ and $ \bsg,\bsg'\in\R^{nN} $.  By  using \eqref{F} and Assumption~\ref{asmp1},  we have
	\begin{equation*}
		\begin{split}
			&(\bx-\bx')^\top \bk \,\col\big((f_i(x_i,{\sigma}_i)-f_i(x_i',{\sigma}_i))_{i\in\cali}\big)\\
			+&(\bx-\bx')^\top \bk \, \col\big((f_i(x_i',{\sigma}_i)-f_i(x_i',{\sigma}_i'))_{i\in\cali}\big)\\
			+&(\bsg-{\bsg}')^\top  \big(({\bsg}-\bt\bx)-(\bsg'-\bt\bx')\big)\\
			\geq &\sum_{i\in\cali} k_i \mu_i \|x_i-x'_i\|^2-(k_i\ell_i+h_i) \|x_i-x'_i\|\|\sigma_i-{\sigma}'_i\|\\
			+&\|\sigma_i-{\sigma}'_i\|^2.
		\end{split}
	\end{equation*}
	As a result, to establish the inequality in  \eqref{ep_mono}, it is sufficient to define $ \epsilon:=\min\{\epsilon_i\} $ where $ \epsilon_i> 0$  satisfies
	\begin{multline*}
		k_i \mu_i \|x_i-x'_i\|^2-(k_i\ell_i+h_i) \|x_i-x'_i\|\|\sigma_i-{\sigma}'_i\|\\+\|\sigma_i-{\sigma}'_i\|^2\geq \epsilon_i\|x_i-x'_i\|^2+\epsilon_i\|\sigma_i-{\sigma}'_i\|^2.
	\end{multline*}
	{Clearly, such $\epsilon_i$ exists providing that}
	\begin{equation*}
		\begin{bmatrix}
			k_i \mu_i & -\frac{(k_i\ell_i+h_i)}{2}\\
			-\frac{(k_i\ell_i+h_i)}{2} & 1
		\end{bmatrix}>0.
	\end{equation*}
	{The above positive definiteness condition holds if and only if}
	\begin{multline*}
		k_i\in \big(\frac{2\mu_i-\ell_ih_i-2\sqrt{\mu_i(\mu_i-\ell_ih_i)}}{\ell_i^2},\\
		\frac{2\mu_i-\ell_ih_i+2\sqrt{\mu_i(\mu_i-\ell_ih_i)}}{\ell_i^2}\big),
	\end{multline*}
	which is equivalent to \eqref{inter}.
	
	$ (ii) $ Let $ {\bsg}=\bone_N\otimes \avg{\bx} $  and $ {\bsg'}=\bone_N\otimes \avg{\bx'} $. By using the definition of $ \avg{\bx} $ we get $ \bsg-\bsg'=\bone_N\otimes \avg{\bx-\bx'} $. Hence,  inequality \eqref{ep_mono}, proven in part $(i)$, becomes
	\begin{equation}\label{ineqKF1}
		\begin{split}
			&(\bx-\bx')^\top \bk \,\col\big((f_i(x_i,\avg{\bx})-f_i(x_i',\avg{\bx'}))_{i\in\cali}\big)\\
			+&(\bone_N\otimes \avg{\bx-\bx'})^\top  \big((\bone_N\otimes \avg{\bx-\bx'})-\bt(\bx-\bx')\big)\\
			\geq &\epsilon \|\bx-{\bx}'\|^2+\epsilon \|\bone_N\otimes \avg{\bx-\bx'}\|^2.
		\end{split}
	\end{equation}
	Let 
	\begin{equation}\label{pi}
		\Pi:=I-{ \frac{1}{N}}\bone_{N}\bone_{N}^\top.
	\end{equation}
	Then,
	\begin{equation*}
		\begin{split}
			(\bone_N\otimes \avg{\bx-\bx'})-\bt(\bx-\bx')= -(\Pi\otimes I_n)\bt(\bx-\bx'),
		\end{split}
	\end{equation*}
	where we have used the equality  $$ \bone_N\otimes \avg{\bx-\bx'}=\frac{1}{N}(\bone_N\otimes\bone_N^\top\otimes I_n)\bt(\bx-\bx') .$$	Therefore, the second term on the left hand side of \eqref{ineqKF1} is zero as $ \bone_N^\top\Pi=0 $, and the proof is complete.\hfill$\square$

\bigskip
\noindent
\textit{Proof of Lemma~\ref{lemNE}.}
		First, we  show that each cost function $ J_i(x_i,\avg{\bx}) $ is $ \eta_i $-strongly convex in $ x_i $ for all $ \bx_{-i}\in \calx_{-i}=\prod_{j\neq i}\calx_j $, and for that, it  suffices  $ f_i(x_i,\avg{\bx}) $  to be $ \eta_i $-strongly monotone in $ x_i $, i.e.,
	\begin{multline*}
	(x_i-x_i')^\top \big(f_i(x_i,\frac{h_i}{N}x_i+\frac{1}{N}\sum_{j\neq i}h_jx_j)-\\
	f_i(x_i',\frac{h_i}{N}x_i'+\frac{1}{N}\sum_{j\neq i}h_jx_j)\big)\geq \eta_i\|x_i-x_i'\|^2,
	\end{multline*}
	for all $ x_i,x_i'\in\calx_i $, $ \bx_{-i}\in\calx_{-i} $, and some $ \eta_i>0 $. We can use Assumption~\ref{asmp1} and rewrite the left hand side of the above inequality as
	\begin{equation}\label{propStrong}
	\begin{split}
	(x_i&-x_i')^\top \big(f_i(x_i,\frac{h_i}{N}x_i+\frac{1}{N}\sum_{j\neq i}h_jx_j)- \\
	f_i&(x_i',\frac{h_i}{N}x_i+\frac{1}{N}\sum_{j\neq i}h_jx_j)\big)+\\
	(x_i&-x_i')^\top \big(f_i(x_i',\frac{h_i}{N}x_i+\frac{1}{N}\sum_{j\neq i}h_jx_j)-\\
	f_i&(x_i',\frac{h_i}{N}x_i'+\frac{1}{N}\sum_{j\neq i}h_jx_j)\big)\geq (\mu_i-\frac{\ell_ih_i}{N}) \|x_i-x_i'\|^2.
	\end{split}
	\end{equation}
	Thus, noting $ \eta_i:= \mu_i-\frac{\ell_ih_i}{N}$ and $ \mu_i>\ell_ih_i $, we conclude that $ J_i(x_i,\avg{\bx}) $ is $ \eta_i $-strongly convex.	By leveraging this property and continuity of $ J_i $ in Assumption~\ref{asmp0}, we conclude that each cost function is radially unbounded with respect to its action \cite[Prop. 1]{baker2016strong}.	Therefore, it follows from \cite[Cor. 4.2]{basar1999dynamic} that the game has an NE $ \bar\bx $ which satisfies
	\begin{equation*}
	\col\big(f_i(\bar x_i,\avg{\bar\bx})_{i\in\cali}\big)=\bze.
	\end{equation*}
	For uniqueness of the NE, we resort to a proof by contradiction. Let $ \bar \bx $ and $ \bx'  $ be two different NE that satisfy the above equality. Choosing $k_i$ according to \eqref{inter} for each $i\in \cali$, by Lemma~\ref{pro1}$ (ii) $,  we find  that
	\begin{multline*}
	(\bar \bx-\bx')^\top \bk \, \col\big((f_i(\bar x_i,\avg{\bar\bx})-f_i(x_i',\avg{\bx'}))_{i\in\cali}\big)=0\\
	\geq \epsilon \|\bar \bx-{\bx}'\|^2,
	\end{multline*}
	{which holds if and only if}  $ \bar \bx = \bx' $, and we reach a contradiction.  This completes the proof.\hfill$\square$

\bigskip
\noindent
\textit{Proof of Lemma~\ref{lemNEset}.}
		Under Assumption~\ref{asmpset}, the game admits an NE if $ J_i(x_i,\avg{\bx}) $ is strictly convex in $ x_i $ for all  $ \bx_{-i}\in \calx_{-i}=\prod_{j\neq i}\calx_j $ \cite[Thm. 4.3]{basar1999dynamic}, which is satisfied as a consequence of Assumption~\ref{asmp1} (see \eqref{propStrong}). 
	We also know that  $ \bar\bx\in \calx $ is an NE if and only if it is a solution of the variational inequality VI$ (\calx,\col\big((f_i(x_i,\avg{\bx}))_{i\in\cali}\big)) $ \cite[Prop. 1.4.2]{facchinei2007finite}. Moreover, since $ \bk\col\big((f_i(x_i,\avg{\bx}))_{i\in\cali}\big) $ is strongly monotone (Lemma~\ref{pro1}$ (ii) $) and $ \calx $ is closed and convex, the variational inequality VI$ (\calx,\bk\col\big((f_i(x_i,\avg{\bx}))_{i\in\cali}\big)) $ has a unique solution $ \bx'\in \calx $ \cite[Thm. 2.3.3]{facchinei2007finite}. Lastly, we need to show that $ \bar\bx $ is unique and equal to $ \bx' $.
	
	Clearly, we have
	\begin{equation*}
	(\bx-\bar{\bx})^\top \col\big((f_i( \bar x_i,\avg{ \bar \bx}))_{i\in\cali}\big)\geq 0,\quad \forall \bx\in\calx,
	\end{equation*}
	which can be rewritten as
	\begin{equation*}
	\sum_{i\in\cali} (x_i-\bar{x}_i)^\top f_i( \bar x_i,\avg{ \bar\bx})\geq 0 ,\quad \forall \bx\in\calx.
	\end{equation*}
	Let for all $ i\in\cali\setminus \{j\} $, we have $ x_i=\bar{x}_i $; thus, by using $ k_j>0 $, the above inequity yields
	\begin{equation*}
	k_j(x_j-\bar{x}_j)^\top f_j(\bar x_j,\avg{\bar \bx})\geq 0.
	\end{equation*}
	By performing the same procedure for the other components of $ \bar \bx $ and re-writing all obtained inequalities into the vector form, we can see that $ \bar \bx $ is the  solution of VI$ (\calx,\bk\col\big((f_i(x_i,\avg{\bx}))_{i\in\cali}\big)) $, i.e., $ \bar \bx= \bx'$. Consequently, since $ \bar \bx $ is an arbitrary solution and $ \bx' $ is unique,  both variational inequality problems have an identical solution, which concludes the proof.	\hfill$\square$
\section*{Appendix B}
\subsection*{On convergence of the Modified NE Seeking Algorithm \eqref{ne_alg_lagrange}:}
Similar to  \eqref{distset}, we can guarantee that for any initial condition  $ ({\bx}(t_0),{\bsg}(t_0),\bpsi(t_0),\bm\lambda(t_0) )\in \calx^1 \times \R^{nN}\times \R^{nN}\times \R^N$, the solution of \eqref{ne_alg_lagrange} is unique and belongs to $ \calx^1 \times \R^{nN}\times \R^{nN} \times \R^N$ for almost all $ t\geq t_0 $. We claim that such a  solution converges to an equilibrium corresponding to the NE of the game. To this end, consider the equilibrium point $ (\bar {\bx},\bar {\bsg},\bar \bpsi,\bar {\bm\lambda}) $ with $ \bar \bsg=\bone_N\otimes \avg{\bar\bx} $, $ (L\otimes I_n)\bar\bpsi= (\Pi\otimes I_n)\bt\bar \bx $, and
\begin{equation}\label{PEV_eqli_proj}
\bze=\Pi_{\calx^1}\big(\bar \bx,-\bk\col\big((f_i(\bar x_i,\avg{\bar\bx}))_{i\in\cali}\big)-(I_N\otimes \bone_n)\bar {\bm\lambda}\big),
\end{equation}
\begin{equation}\label{PEV_eqli_x2}
\\
\bze=(I_N\otimes \bone_n^\top)\bar \bx-\bm\gamma.
\end{equation}
The second equality implies that $ \bar \bx \in \calx^2= \prod_{i\in \cali}\calx_i^2$. Employing Moreau's decomposition theorem and \eqref{PEV_eqli_proj}, we can perform an analogous analysis to the proof of Theorem~\ref{the:convSet} and conclude that $ \bar \bx $ is the solution of VI$ (\calx^ 1,\bk\col\big((f_i( x_i,\avg{ \bx}))_{i\in\cali}\big)+(I_N\otimes \bone_n)\bar {\bm\lambda}) $. This means that $ \bar \bx $ is the solution of the following optimization problem \cite[Eq. 1.2.1]{facchinei2007finite}
\begin{equation*}
\min_{y\in\calx^1}y^\top (\bk\col\big((f_i( \bar x_i,\avg{\bar  \bx}))_{i\in\cali}\big)+(I_N\otimes \bone_n)\bar {\bm\lambda}).
\end{equation*}
Let $ g_i^t(\bx):=\col(x_i^t-\bar x_i,-x_i^t)$, $ g_i(\bx):=\col\big((g_i^t(\bx))_{t\in\calt}\big) $, and $ g(\bx):=\col\big((g_i(\bx))_{i\in\cali}\big) $; then we see that $ g(\bx)\leq 0 $ represents the set $ \calx^1 $. Therefore, there exist $ \mu_i^t\in\R^2 $ that satisfy the following KKT conditions

\begin{alignat*}{2}
0&=\bk\col\big((f_i( \bar x_i,\avg{\bar  \bx}))_{i\in\cali}\big)&&+(I_N\otimes \bone_n)\bar {\bm\lambda}\\
&  &&+\sum_{i\in\cali}\sum_{t\in\calt} \fpart{}{x}g_i^t(\bx)^\top \mu_i^t\\
0&\leq \bm{\mu} \perp g(\bx)\leq 0,
\end{alignat*}
where $ \bm{\mu}:=\col\big((\mu_i)_{i\in\cali}\big) $ with $ \mu_i:=\col\big((\mu_i^t)_{t\in\calt}\big) $. Considering the above equations together with \eqref{PEV_eqli_x2}, we conclude from \cite[Prop. 1.3.4(b)]{facchinei2007finite} that $ \bar \bx $ is the solution of VI$ (\calx,\bk\col\big((f_i( x_i,\avg{ \bx}))_{i\in\cali}\big)) $, and in turn, it is the NE of the game. Convergence of the algorithm is similar to Theorem~\ref{the:convSet}.

\end{document}